\documentclass[11pt,reqno]{amsart}

\usepackage{psfrag}
\usepackage{a4wide}
\usepackage[scaled=.90]{helvet}
\usepackage{natbib}
\usepackage{amsbsy,amsfonts,amsmath,amssymb}
\usepackage{epsfig,graphicx,rotating,subfigure,enumerate,a4wide}

\newcommand{\R}{\mathbb{R}}
\newcommand{\N}{\mathbb{N}}
\newcommand{\Rb}{\overline{\R}}

\newcommand{\Ls}[1]{\text{L}^{#1}}
\renewcommand{\L}[1]{\text{L}^{#1}(\Omega)}

\newcommand{\BV}{\text{BV}(\Omega)}
\newcommand{\TV}{\text{BV}}
\renewcommand{\H}[1]{\text{H}^{#1}(\Omega)}

\newcommand{\D}{\text{D}}

\newcommand{\diff}{\textnormal{\,d}}
\newcommand{\dx}{\diff x}

\newcommand{\diw}[1]{\textnormal{div}\left(#1\right)}

\DeclareMathOperator*{\argmin}{\textnormal{argmin}}

\newcommand{\Cov}[2]{\mathbf{Cov}\left( #1, #2 \right)}

\newcommand{\Prob}{\mathbb{P}}

\newcommand{\med}{\text{med}}

\newcommand{\supp}{\textnormal{supp}}
\newcommand{\eps}{\varepsilon}

\renewcommand{\ker}{\textnormal{ker}}
\newcommand{\ran}{\textnormal{ran}}
\newcommand{\set}[1]{\left\{ #1 \right\}}
\newcommand{\abs}[1]{\left| #1 \right|}
\newcommand{\norm}[1]{\left\| #1 \right\|}
\newcommand{\inner}[2]{\left\langle #1, #2 \right\rangle}

\newcommand{\bigo}{\mathcal{O}}
\newcommand{\ra}{\rightarrow}

\newcommand{\spa}{\textnormal{span}}
\newcommand{\err}{\textnormal{err}}

\usepackage{ifthen}
\newcommand{\hideandshow}[2]{%
\ifthenelse{\isundefined{\shortversion}}{#1}{#2}}

\theoremstyle{plain}

\newtheorem{theorem}{Theorem}[section]
\newtheorem{proposition}[theorem]{Proposition}

\newtheorem{assumption}[theorem]{Assumption}

\theoremstyle{definition}

\newtheorem{lemma}[theorem]{Lemma}

\newtheorem{definition}[theorem]{Definition}
\newtheorem{corollary}[theorem]{Corollary}
\newtheorem{example}[theorem]{Example}
\newtheorem*{example*}{Example}

\newtheorem*{dfn*}{Definition}
\newtheorem*{alg*}{Algorithm}

\theoremstyle{remark}
\newtheorem{remark}{Remark}[section]

\usepackage{amsaddr}

\begin{document}

\title[Shape Constrained Regularisation]{Shape Constrained Regularisation by
Statistical Multiresolution for Inverse Problems: Asymptotic Analysis}

\author{Klaus Frick}  
\address{Institute for Mathematical Stochastics\\
University of G{\"o}ttingen\\
Goldschmidtstra{\ss}e 7, 37077 G{\"o}ttingen}
\email{frick@math.uni-goettingen.de}
\thanks{Correspondence to frick@math.uni-goettingen.de }

\author{Philipp Marnitz}
\address{Institute for Mathematical Stochastics\\
University of G{\"o}ttingen\\
Goldschmidtstra{\ss}e 7, 37077 G{\"o}ttingen}
\email{stochastik@math.uni-goettingen.de}

\author{Axel Munk}
\address{Institute for Mathematical Stochastics\\
University of G{\"o}ttingen\\
Goldschmidtstra{\ss}e 7, 37077 G{\"o}ttingen\\
and}
\address{Max Planck Institute for Biophysical Chemistry \\
Am Fa{\ss}berg 11, 37077 G{\"o}ttingen} 
\email{munk@math.uni-goettingen.de}

 
 
\begin{abstract}
This paper is concerned with a novel regularisation technique for solving linear
ill-posed operator equations in Hilbert spaces from data that is corrupted by
white noise. We combine convex penalty functionals with extreme-value
statistics of projections of the residuals on a given set of sub-spaces in the
image-space of the operator. We prove general consistency and convergence rate
results in the framework of Bregman-divergences which allows for a vast range of
penalty functionals.

Various examples that indicate the applicability of our approach will be
discussed.  We will illustrate in the context of signal and image
processing that the presented method constitutes a locally adaptive
reconstruction method.
\end{abstract}
 
\maketitle

\section{Introduction}\label{intro} 

In this paper, we are concerned with the solution of the equation
\begin{equation}\label{intro:lineqn}
  Ku = g,
\end{equation}
where $K:U\ra V$ is a linear and bounded operator mapping between two
Hilbert-spaces $U$ and $V$. Equations of type \eqref{intro:lineqn} are  called
\emph{well-posed} if for given $g\in V$ there exists a unique solution 
$u^\dagger\in U$ that depends continuously on the right-hand side $g$.  If one
of these conditions is not satisfied, the problem is called \emph{ill-posed}. In  the
case of ill-posedness, arbitrary small deviations in the right hand side $g$
may lead to useless solutions $u$ (if solutions exist). These
deviations are commonly modelled as random. They are due to
indispensable numerical errors as well as to the random nature of the
measurement  process itself. \emph{(Statistical) regularisation methods} aim
at computing stable approximations of true solutions $u$
from a (statistically)  perturbed signal $g$.
 
In this paper we assume that  we are given the observation
\begin{equation}\label{intro:noisemodel}
  Y = Ku^\dagger + \sigma \eps.
\end{equation}
Here, $\sigma>0$ denotes the noise-level and $\eps:V\ra \Ls{2}(X,
\mathfrak{A}, \Prob)$ a Gaussian white noise process, i.e. $\eps$ is linear and
continuous and for all $v,w \in V$ one has
\begin{equation*}
  \eps(v) \sim \mathcal{N}(0,\norm{v}^2)\quad\text{ and }\quad
  \Cov{\eps(v)}{\eps(w)} = \inner{v}{w},
\end{equation*}    
where $\mathcal{N}(\mu, \sigma^2)$ denotes the normal distribuion with
expectation $\mu$ and variance $\sigma^2$. The white noise model
\eqref{intro:noisemodel} is very common in the theory of statistical inverse
problems \citep[see e.g.][]{ frick:BisHohMunRuy07, frick:Cav08, frick:CavTsy02,
frick:ChoIbrKha99, frick:CohHofRei04, frick:JohSil91, frick:NusPer99}
and it can be regarded as reasonable approximation to models relevant for many areas of
applications.  A statistical regularisation method amounts to compute an
estimator $\hat u = \hat u (\sigma)$ given the data $Y$
 in \eqref{intro:noisemodel} such that $\hat u(\sigma)\ra u^\dagger$ (in an
 appropriate sense) as $\sigma \ra 0^+$ .
 
The simplest case covered by Model \eqref{intro:noisemodel} is classical
nonparametric regression and its amplitude of applications. Here, $U$ and $V$
are suitable function spaces where it is assumed that $U$ can be continuously
embedded into $V$. $U$ models the smoothness of the true signal $u^\dagger$ and
$K$ is the embedding operator $K:U\hookrightarrow V$ (cf.
\cite{frick:BisHohMunRuy07}). More sophisticated examples for $K$ arise in
imaging, when blurring induced by the recording optical systems is modelled as a
convolution with a kernel $k(x-y)$. Beyond convolution, different operators $K$
occur in various other applications (see e.g. \cite{frick:BerBoc98,
frick:EHN96,frick:SchGraGroHalLen09}).

Due to the broad area of applications, the literature on statistical
regularisation methods is vast. We only give a few, selective references:
penalised least-squares estimation  (that includes Tikohonov-Philipps and
maximum entropy regularisation) \citep{frick:BisHohMun04, frick:Sul86,
frick:Wah77}, wavelet based methods \citep{frick:Don93, frick:Don95,
frick:Joh99, frick:JohKerPicRai04, frick:KerKyrPenPet10}, estimation in
Hilbert-scales \citep{frick:BisHohMunRuy07, frick:GolPer03, frick:MaiRuy96,
frick:MatPer01, frick:MatPer03, frick:MatPer03a} and regularisation by
projection \citep{frick:CavGolPicTsy02, frick:CavTsy02, frick:CohHofRei04, frick:HofRei08,
frick:MatPer01} to name but a few.
    
In this work, we follow a different route and study a variational estimation
scheme that defines estimators $\hat u$ as solutions of
\begin{equation}\label{intro:opt}      
  \inf_{u\in U} J(u)\quad \text{ subject to }\quad  T_N(\sigma^{-1}(Y - Ku))
  \leq q_N(\alpha).
\end{equation}  
Here, $J$ is a \emph{convex regularisation functional} that is supposed to
measures the regularity of candidate estimators $u\in U$ and $T_N$ is a
\emph{data fidelity term} on $V$ that measures the deviation of the data $Y$
and the estimated image $Ku$. In this work we consider fidelity measures $T_N$
of the form 
\begin{equation}\label{intro:mrstatrough} 
T_N(v) = \max_{1\leq n\leq N} \mu_n(v),\quad \text{ for }v \in V.
\end{equation}
The functions $\mu_n:V\ra \R$ are designed to be sensitive to
non-random structures in $v$. We will refer to 
\eqref{intro:mrstatrough} as \emph{multiresolution statistic (MR-statistic)}
and to corresponding solutions of the optimisation problem \eqref{intro:opt} as
\emph{statistical multiresolution estimators (SMRE)}. 

The parameter $q_N(\alpha)$ in \eqref{intro:opt} is chosen to be the
$(1-\alpha)$-quantile of the statistic $T_N(\eps)$ and governs the trade-off
between data-fit and regularity. Hence the
\emph{admissible region} 
\begin{equation}\label{intro:admis}
  \mathcal{A}_N(\alpha) = \set{u \in U ~:~ T_N(\sigma^{-1}(Y
  - Ku))\leq q_N(\alpha)}
\end{equation}
constitutes a $(1-\alpha)$-confidence region for a solution $\hat u$ of
\eqref{intro:opt}, i.e. a region which covers the true solution $u^\dagger$ with
probability $1-\alpha$ at least. This gives the estimation procedure
\eqref{intro:opt} a precise statistical interpretation: Since for each solution
$u^\dagger$ of \eqref{intro:lineqn} one has $u^\dagger \in
\mathcal{A}_N(\alpha)$ with probability at least $1-\alpha$ it follows from \eqref{intro:opt} that
\begin{equation*}
  \Prob\left( J(\hat u) \leq J(u^\dagger) \right) \geq 1-\alpha.
\end{equation*}
Summarizing, regularisation methods of type \eqref{intro:opt} pick among all
estimators $\hat u$ for which the distance between $K\hat u$ and the data $Y$
does not exceed the threshold value $q_N(\alpha)$ one with largest regularity.
The probability that this particular estimator is more regular than any solution of
\eqref{intro:lineqn} is bounded from below by $1-\alpha$ . This is in contrast
to many other regularisation techniques where regularisation parameters merely
govern the trade-off between fit-to-data and smoothness and do not allow such an
interpretation. (In the case of wavelet-thresholding, this property was studied
in \cite{frick:Don95a})

Whereas most of the literature is concerned with the proper choice of the
regularisation functional $J$, in this work we will discuss the issue of the
data fidelity term $T_N$. We claim that from a statistical perspective the
choice of $T_N$ is of equal importance as the choice of $J$. 

In Definition \ref{gen:mrstat} below we will delimit a class of feasible
functions for  $\mu_1,\ldots,\mu_N$ in \eqref{intro:mrstatrough}. However, in
order to make ideas clear (and also to justify the notion ``multiresolution''), we
will start with a simple, yet illustrative example: Let $G\subset [0,1]^d$ be
the equi-spaced grid of points in the unit cube and assume that $V$ consists of
all real valued functions $v:G\ra\R$. Moreover, let $\set{S_1,S_2,\ldots,S_N}$
be a sequence of non-empty subsets of $G$. We define for $n\in\N$ and $v\in V$
the \emph{local average function} $\mu_n(v) = \abs{\sum_{\nu\in V} v_\nu}\slash
\sqrt{\#{S_n}}$, where $\#{S_n}$ denotes the number of grid-points in $S_n$.
Thus, the MR-statistic $T_N$ reads as
\begin{equation}\label{intro:mrstatspecial}
T_N(\sigma^{-1}(Y-Ku)) =  \max_{1\leq n\leq N}\frac{1}{\sqrt{\#{S_n}}
}\abs{\sum_{\nu\in S_n} \sigma^{-1}(Y-Ku)_\nu }.
\end{equation}
In other words,  the statistic $T_N$ returns the largest local average of the
residuals $\sigma^{-1}(Y-Ku)$ over the sets $S_1,\ldots,S_N$.  Under the
hypothesis that $u^\dagger$ is the true solution of \eqref{intro:lineqn}, we
have that $T_N(\sigma^{-1}(Y-Ku^\dagger)) = T_N(\eps)$ does not exceed the
threshold $q_N(\alpha)$ with probability $1-\alpha$ at least. Recall that
$\eps$ is a white noise process and hence ``oscillates around zero'' as an effect of which
the quantile values $q_N(\alpha)$ are relatively small due to cancellations in
the sums in \eqref{intro:mrstatspecial}. If, however, $u$ is wrongly specified the
residual $Y - Ku$  contains a non-random signal which may happen to be covered
by a set $S_{n_0}$. As an effect the local average over $S_{n_0}$ - and thus
also the statistic $T_N(\sigma^{-1}(Y-Ku))$ - becomes relatively large and $u$
lies outside the admissible domain of the optimisation problem
\eqref{intro:opt}.

The choice of the system $\set{S_1,\ldots,S_N }$ is subtle, since it should not
miss any non-random information in the residual, if present. Put differently, it
encodes a priori information on where one expects to encounter non-random
behavior in the residuals of \emph{any possible} estimator $\hat u$.  Thus,
$T_N$ would be most sensible against a large variety of signals $\hat u$, if we employ
a large number $N$ of overlapping sets $S_n$ that cover $G$. This approach,
however, turns \eqref{intro:opt} into an optimisation problem with a huge number
of constraints which is hard to tackle numerically (this is treated separately
in \cite{frick:FriMarMun11}). Besides these numerical difficulties, there is
also a statistical limitation which will be a major issue  to be discussed in
this paper: If the dictionary $\set{S_1,\ldots,S_N }$ is too large (in the sense
of a metric entropy), the asymptotic distribution of $T_N$ will degenerate.  In
practical situations, a priori knowledge on the true solution of \eqref{intro:lineqn} can be used in order to
design dictionaries whose  entropy guarantees a non-degenerate limit of $T_N$
and in addition allows to derive rates of convergence of the SMRE to the true
signal. A similar comment applies to the choice of the regularisation functional
$J$  which models a priori information on the regularity of the true solution.

As a consequence, the MR-statistic $T_N$ plugged in into \eqref{intro:opt} plays
the role of a \emph{shape constraint} and the resulting estimation method is
capable of adapting the amount of regularization in a \emph{locally adaptive}
manner. Put differently, our approach offers a general methodology to
\emph{localise} any global convex regularisation functional in order to obtain
spatial adaption. This is in contrast to global data fidelity terms such as the
widely used squared $2$-norm fidelity (or any other $p$-norm, $p\geq 1$ for that
matter) that do not allow for adaptation to local structures. This is
illustrated in the following example:

\begin{example}\label{intro:ex} 
Assume that $U = V = \R^n$  with
$n=1024$ and let $K:U\ra V$ be the identity
operator, i.e. \eqref{intro:noisemodel} can be rewritten into the simple
nonparametric regression model
\begin{equation*}
Y_i = u_i^\dagger + \sigma \eps_i,\quad i=1,\ldots,n
\end{equation*}
with $\eps_1,\ldots,\eps_n$ i.i.d. standard normal random variables. 
The signal $u^\dagger \in U$ and the data $Y$ according to
\eqref{intro:noisemodel} with $\sigma = 0.05$ are depicted in Figure
\ref{intro:figdata}. 

\begin{figure}[h!]
\begin{center}
\includegraphics[width = 0.27\textwidth]{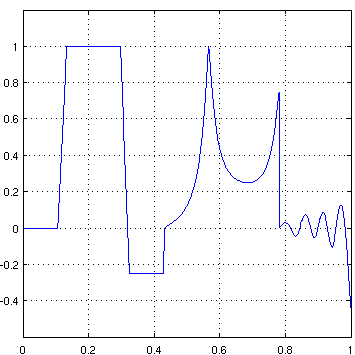}
\hspace{0.02\textwidth}
\includegraphics[width = 0.27\textwidth]{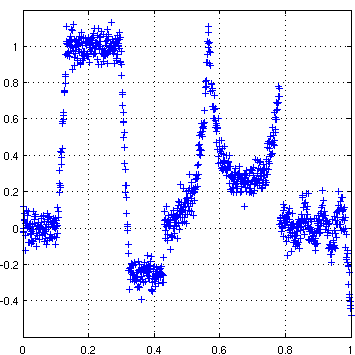}
\caption{True signal $u^\dagger$ (left) and data $Y$
(right).}\label{intro:figdata}
\end{center}
\end{figure}

The signal $u^\dagger$ exhibits kinks, jumps, peaks and smooth
portions simultaneously which makes estimation a delicate matter. For
example, the regularisation functional
\begin{equation}\label{intro:h1penalty}
J(u) = \frac{1}{n}\sum_{k=1}^{n-1}\abs{u_{k+1} - u_k}^2.
\end{equation}
appears to be well suited to recover at least the smooth parts of the signal,
however, with a tendency to ``smear out'' edges, peaks and kinks. In the
following we will show how this deficiency can be repaired by \emph{localising}
$J$ by means of MR-statistics. To this end we will compute SMREs, solutions of
\eqref{intro:opt} that is, with $J$ as in \eqref{intro:h1penalty}. 

Before we do so, we start with reconstructing $u^\dagger$ by the usual
``global'' approach for the purpose of comparison. We compute a $J$-penalized
least squares estimator $\hat u_2$, i.e. the solution of
\begin{equation}\label{intro:penleastsquare}
\min_{u\in \R^n}\frac{1}{n} \sum_{l=1}^n \abs{Y_l - u_l}^2 +
\frac{\lambda}{n} \sum_{l=1}^{n-1}\abs{u_{l+1} - u_l}^2.
\end{equation} 
Here, the proper selection of smoothness amounts to a proper choice of the
parameter $\lambda > 0$. It is instructive to rewrite
\eqref{intro:penleastsquare} in a slightly different form, such that the
relationship to \eqref{intro:opt} becomes obvious: To each $\lambda > 0$ there
corresponds a threshold value $q = q(\lambda)$, such that $\hat u_2$ is a
solution of
\begin{equation}\label{intro:penleastsquarealt}
\min_{u\in\R^n}\frac{1}{n}\sum_{k=1}^{n-1}\abs{u_{k+1} - u_k}^2\quad\text{ s.t.
}\quad \frac{1}{n}\sum_{l=1}^n \abs{Y_l - u_l}^2 \leq q. 
\end{equation}
The first
three panels in the upper row of Figure \ref{intro:figsmre} depict solutions
$\hat u_2$ for $q=25,43$ and $50$. The choice $q=43$ yields the visually best
result, however it becomes immediately clear that there are under- and
oversmoothed parts in the reconstruction. The latter becomes undeniably visible
in the qq-plot of the residual $Y - \hat u_2$ (lower row) which indicates that
there is a significant amount of outliers. Note, that less oversmoothing,
i.e. fewer outliers in the residuals, can only be achieved at the cost of more
artefacts in the reconstruction (by decreasing $q$) and vice versa fewer
artefacts only by accepting severe oversmoothing (by increasing $q$). This is due to the fact that
each residual value $Y_l - u_l$ ($1\leq l\leq n$) contributes equally to
the quadratic fidelity in \eqref{intro:penleastsquare} (or likewise in
\eqref{intro:penleastsquarealt}) \emph{independent of its spatial position}.

%  We finally remark that for jump recovery other regularization functionals $J$
%  usually are proposed to be combined with the quadratic fidelty, above all the
%  total variation semi-norm. However, this functional provokes piecewise constant
%  reconstructions which is
% clearly undesirable for smooth signals (``staircaising'' effect).

\begin{figure}[h!]
\begin{center}
\includegraphics[width = 0.242\textwidth]{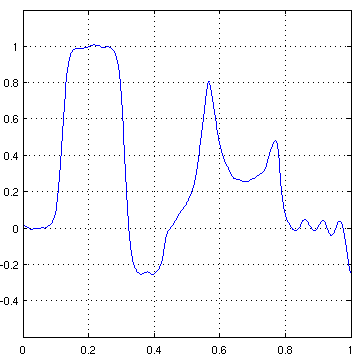}
\includegraphics[width = 0.242\textwidth]{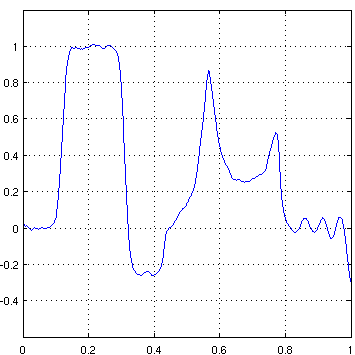}
\includegraphics[width = 0.242\textwidth]{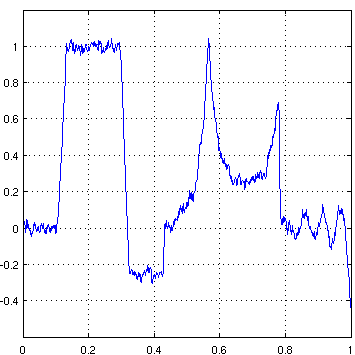}
\includegraphics[width = 0.242\textwidth]{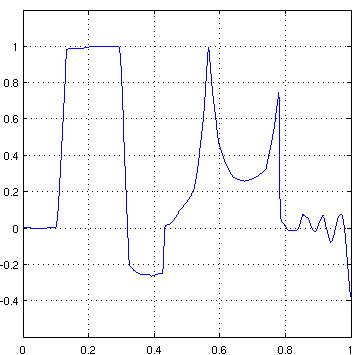}

\includegraphics[width = 0.242\textwidth]{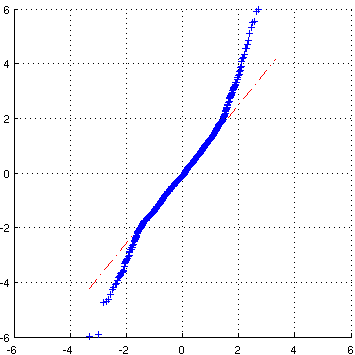}
\includegraphics[width = 0.242\textwidth]{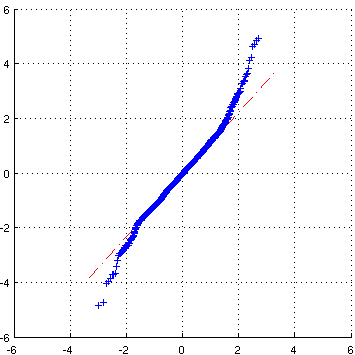}
\includegraphics[width = 0.242\textwidth]{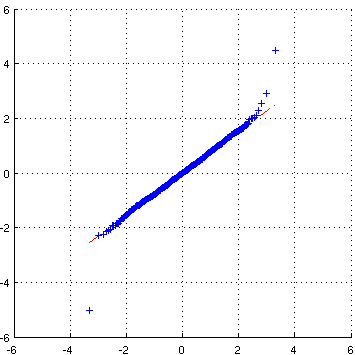}
\includegraphics[width = 0.242\textwidth]{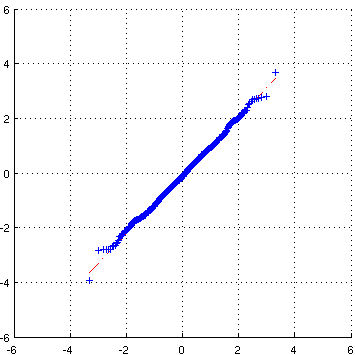}
\caption{Upper row: Global estimators $\hat u_2$ for $q=25, 43$ and $50$
and SMRE $\hat u_\text{SMRE}$. Lower row: Corresponding qq-plots of the
standardised residuals against standard normal}.\label{intro:figsmre}
\end{center}
\end{figure}

To overcome this obvious ``lack of locality'' , we compute solutions of
\eqref{intro:opt} where we employ the MR-statistic in
\eqref{intro:mrstatspecial} as fidelity measure. To be more precise, we choose
the sets $\set{S_1,\ldots,S_N}$ to consist of all discrete intervals of the type
$\set{i,\ldots,j}\slash n$ with $1\leq i<j \leq n$ and $j - i \leq 20$ (i.e. $N
= 20.290$). Put differently, the SMRE $\hat u_\text{SMRE}$ is a solution of the
convex optimisation problem
\begin{equation*}
\min_{u\in\R^n}\frac{1}{n}\sum_{i=1}^{n-1}\abs{u_{k+1} - u_k}^2\quad\text{ s.t.
}\quad \max_{\substack{1\leq i < j\leq n \\ j-i \leq 20}} \frac{1}{\sqrt{j-i+1}}
\abs{\sum_{l=i}^j Y_l - u_l}\leq q. 
\end{equation*}
For the computation of $\hat u_\text{SMRE}$  in the rightmost panel of Figure
\ref{intro:figsmre} we set $q = q_N(\alpha)=2.9$ which corresponds to a small
value of $1-\alpha\approx 0.01$ in order to avoid oversmoothing. The value of
$\alpha$ was determined by simulations of the statistic $T_N(\eps)$.  Indeed,
the result is visually appealing: The kinks, jumps and peaks are strikingly well
recovered,  both in location and height and the smooth parts of the signal
exhibit no artefacts. Also the corresponding qq-plot confirms  that there are
hardly any outliers in the residuals $Y - \hat u_\text{SMRE}$,  which indicates
that oversmoothing is limited to a reasonable amount. Again, this is all the
more remarkable as the regularisation functional $J$ is known to usually blur
edges, peaks and kinks.

Summarising, it becomes evident that the SMRE approach outperforms the standard
method that employs the global quadratic fidelity. In particular, this example
shows that plugging in the MR-statistic $T_N$ into \eqref{intro:opt} results in
an estimation scheme that regularises in a \emph{locally adaptive}
manner. Aside to the specific choice \eqref{intro:h1penalty} any other convex
regularisation functional $J$ can be ``localised'' in this way, of course, as
for example the total variation semi-norm
\begin{equation*}
J(u) = \frac{1}{n}\sum_{k=1}^{n-1}\abs{u_{k+1} - u_k}.
\end{equation*}
It has turned out, however, that for the present example \eqref{intro:h1penalty}
is preferable since it accounts well for the smooth parts in the signal, whereas
it is well known and also visible that the total variation penalty induces an
undesired ``staircasing'' effect. This is illustrated in Figure \ref{intro:figsmre_tv}, where global estimators $\hat u_2$
and the SMRE $\hat u_{\text{SMRE}}$ are depicted. 

\begin{figure}[h!]
\begin{center}
\includegraphics[width = 0.242\textwidth]{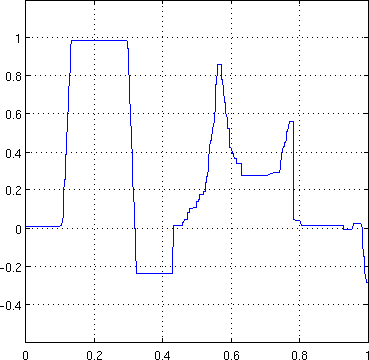}
\includegraphics[width = 0.242\textwidth]{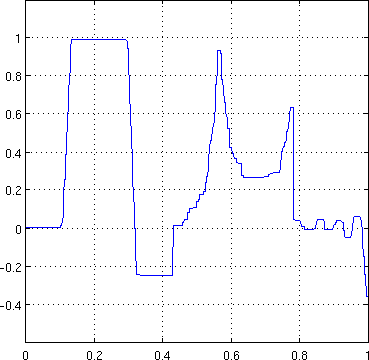}
\includegraphics[width = 0.242\textwidth]{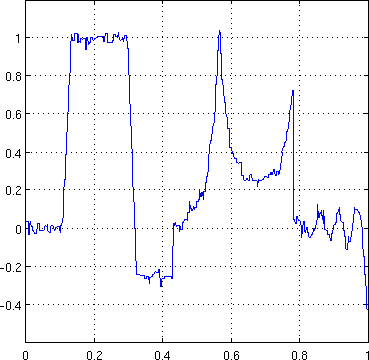}
\includegraphics[width = 0.242\textwidth]{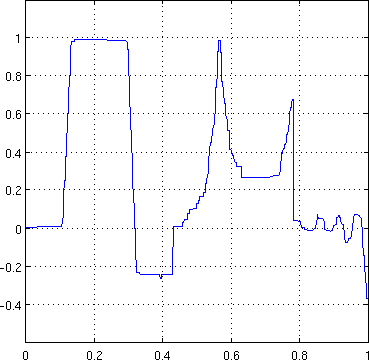}
 
\caption{Global estimators $\hat u_2$ for $q=29, 33$ and $36$
and SMRE $\hat u_\text{SMRE}$ w.r.t. the total variation
penalty.}.\label{intro:figsmre_tv}
\end{center}
\end{figure}

We finally remark, that all estimators in this example were computed by a
alternating direction method of multipliers (ADMM) as developed in
\cite{frick:FriMarMun11} and its details will not be discussed here. In
\cite{frick:FriMarMun11} also simulation studies are performed giving
quantitative evidence of the good performance of our method (see also Examples
\ref{tv:ex_denois} and \ref{tv:ex_deconv}).
\end{example} 

The regularisation scheme \eqref{intro:opt} with the MR-statistic $T_N$ as in
\eqref{intro:mrstatspecial}  was studied in
\cite{frick:DavKovMei09} for the specific case  of non-para\-metric regression
in one space dimension and the total-variation  semi-norm as regularisation
functional $J$.  In this paper we will show that the general formulation in
\eqref{intro:opt} reveals the SMRE as a powerful regularisation method far
beyond this situation: It can be extended to space dimensions larger than one as
well as to inverse problems with general $K$ as in \eqref{intro:noisemodel}
including deconvolution problems. Furthermore, we present very general
consistency and convergence rates results for SMRE  in the context of statistical inverse problems and discuss
their impact on particular applications. To our best knowledge, results of this
type have never been obtained before. It is necessary to assume additional
regularity of the true solution of \eqref{intro:lineqn} in order to come up with
convergence rates results. In the context of inverse problems, this is usually
done by imposing  \emph{source conditions}. These determine
smoothness classes of solutions for \eqref{intro:lineqn} that guarantee risk bounds and
fast convergence of the estimator to the true signal. In this work we study the
standard source condition \cite{frick:BO04} used in the framework of
Bregman-divergences that yield for each penalty functional $J$ in \eqref{intro:opt} \emph{one} specific
smoothness class.  The formulation of conditions that give optimal convergence
rates in a \emph{scale} of smoothness classes for a general but \emph{fixed} $J$
will not be treated in this work (cf. \cite{frick:Fle11} and references
therein).
 
This paper is organised as follows. After reviewing some basic definitions from
convex analysis and the theory of inverse problems in Section \ref{defs} we
develop in Section \ref{gen:multiressec} a general scheme for estimation
for the statistical inverse problem \eqref{intro:noisemodel} based on the convex
optimisation problem
\eqref{intro:opt}. In Section \ref{gen:conssec}
we then prove consistency and convergence rate results in terms of the
Bregman-divergence w.r.t. the regularisation functional $J$. In Section
\ref{simple} we study the performance of the so constructed estimators for
various examples, as the Gaussian sequence model (Section \ref{simple:indep})
and linear inverse regression problems (Section \ref{lutz}). In Section \ref{tv}
we investigate the particular situation when the regularisation functional $J$
is chosen to be the total-variation semi-norm, which has a particular appeal for
imaging problems. Finally, some examples that illustrate the notions of
source-condition and Bregman-divergence are given in Appendix \ref{app:breg} and
the proofs of the main results as well as some auxiliary lemmata are collected
in Appendix \ref{app}.

%%%
\section{Basic Definitions}\label{defs}
%%% 

In this section we summarise some relevant definitions and assumptions needed
throughout the paper. 
 
\begin{assumption}\label{defs:basicass} 
\begin{enumerate}[(i)]
  \item $U$ and $V$ denote separable Hilbert spaces. The norms on
  $U$ and $V$ are not further
  specified, and will be always denoted by $\norm{\cdot}$, since the
  meaning is clear from the context.
  \item Let $J:U \ra \Rb$ be a convex functional from $U$ into the extended
  real numbers $\Rb=\R \cup \set{\infty}$.
  The domain of $J$ is defined by
  \begin{equation*}
    D(J) = \set{ u \in U : J(u) \neq \infty}.
  \end{equation*}
  $J$ is called \emph{proper} if $D(J) \neq \emptyset$ and $J(u) >
  -\infty$ for all $u\in U$. Throughout this paper $J$ denotes a convex,
  proper and lower semi-continuous (l.s.c.) functional with dense domain $D(J)$.
  \item $K : U \to V$ is a linear and bounded operator. By
  $\ran(K) = K(U)$ we denote the range of $K$ and by $K^*:V \ra U$ the adjoint
  operator of $K$.
\end{enumerate}
\end{assumption}
In the course of this paper we will frequently make use of tools from convex
analysis. For a standard reference see \cite{frick:ET76}.
\begin{itemize}
  \item The sub-differential (or generalised derivative)  $\partial J(u)$ of $J$
  at $u$ is the set of all elements $p \in U$ satisfying
  \begin{equation*}
    J(v) - J(u) - \inner{p}{v-u} \geq 0\quad\text{ for all }v\in U.
  \end{equation*}
  The \emph{domain} $D(\partial J)$ of the sub-differential consists of all
  $u\in U$ for which $\partial J (u) \neq \emptyset$.
  
  \item We will prove consistency of estimators with respect to the
  \emph{Bregman-divergence}. For $u \in D(J)$ the Bregman-divergence of $J$
  between $u$ and $v$ is defined by
  \begin{equation*}
    D_J(v,u) = J(v) - J(u) - J'(v)(v-u)
  \end{equation*}
  where $J'(v)(v-u)$ denotes the directional derivative of $J$ at $v$ in
  direction $v-u$. The directional derivative is defined as
  \begin{equation*}
    J'(v)(w) = \lim_{h\ra 0^+}\frac{J(v+hw) - J(v)}{h}.
  \end{equation*}
  and is well defined for convex functions (possibly with values in $[-\infty,
  \infty]$).
  \item For $u \in D(\partial J)$ the  Bregman-divergence of $J$ between
  $u$ and $v$ w.r.t. $\xi \in \partial J(u)$ is defined as
  \begin{equation*}
    D_J^\xi(v,u) = J(v) - J(u) - \inner{\xi}{v-u}.
  \end{equation*}
  The following basic estimates hold
  \begin{equation*}
    0\leq D_J(v,u) \leq D_J^\xi(v,u),\quad \text{ for all }\xi\in\partial J(u).
  \end{equation*}
\end{itemize}
    
\begin{remark}\label{defs:bregrem}
Clearly, the Bregman-divergence does not define a (quasi-)metric on $U$: It is
non-negative but in general it is neither symmetric nor satisfies the triangle
inequality. The big advantage, however, of formalising asymptotic results
w.r.t. to the Bregman-divergence (such as consistency or convergence rates) for
estimators defined by a variational scheme of type \eqref{intro:opt}, is the
fact, that the regularising properties of the used penalty functional $J$ are
incorporated automatically. If, for example, the functional $J$ is slightly
more  than strictly convex, it was shown in \cite{frick:R04} that convergence
w.r.t. the Bregman-divergence already implies convergence in norm. If, however,
$J$ fails to be strictly convex (e.g. if it is of linear growth) it is in
general hard to establish norm-convergence results but convergence results
w.r.t. the Bregman-divergence, though weaker, may still be at hand.  In 
the Appendix \ref{app:breg} we compute the Bregman-divergence for some particular
choices of $J$.

The concept of Bregman-divergence in optimisation was introduced in
\cite{frick:Bre67} and has recently attracted much attention e.g. in the
inverse problems community \citep[see][]{frick:BO04,frick:Csi91,frick:FriSch10}
or in statistics and machine learning \citep{frick:ColSchSin02,
frick:LafPiePie97, frick:ZhaYu05}.
\end{remark}

Next, we introduce different classes of solutions for Equation
\eqref{intro:lineqn} discussed in this paper.

\begin{definition}\label{defs:solution}
\begin{enumerate}[(i)]
  \item Let $u \in D(J)$ be a solution of \eqref{intro:lineqn}. Then $g$ is
  called \emph{attainable}.
  \item An element $u \in D(J)$ is called \emph{$J$-minimising solution} of
  \eqref{intro:lineqn}, if $u$ solves \eqref{intro:lineqn} and
  \begin{equation*}
    J(u) = \inf \set{ J(\tilde u) ~:~ K\tilde u = g}.
  \end{equation*}
  \item Let $g\in V$ be attainable. An element $p\in V$ is called a
  \emph{source element} if there exists a $J$-minimising solution $u$ of
  \eqref{intro:lineqn}  such that
  \begin{equation}\label{defs:solution:source}
    K^*p \in \partial J(u).
  \end{equation}
  Then, we say that $u$ \emph{satisfies the source condition}
  \eqref{defs:solution:source}.
\end{enumerate}
\end{definition}

It is well-known in the theory of inverse problems with deterministic noise 
\citep[see][]{frick:EHN96} that the source condition
\eqref{defs:solution:source} is sufficient for establishing convergence rates for regularisation methods. It can
be understood as a regularity condition for $J$-minimising solutions of Equation
\eqref{intro:lineqn}. Put differently, for each regularisation functional $J$
and each operator $K$, the source condition \eqref{defs:solution:source}
characterises \emph{one particular} smoothness-class of solutions for \eqref{intro:lineqn} for which fast
reconstruction is guaranteed. We clarify the notions
\emph{Bregman-divergence} and \emph{source condition} by some examples in
Appendix \ref{app:breg}. 

Under fairly general conditions existence of $J$ minimising solution can be
guaranteed. We formalise these conditions in the following result, however, we
omit the proof since it is standard in convex analysis \citep[see][Chap. II
Prop. 2.1]{frick:ET76}.

\begin{proposition}\label{defs:exjmin}
Let $g\in V$ be attainable and assume that for all $c\in \R$ the sets
\begin{equation}\label{defs:comset}
  \set{ u \in U ~:~ \norm{Ku} + J(u) \leq c}
\end{equation}
are bounded in $U$. Then, there exist a $J$-minimising solution of
\eqref{intro:lineqn}.
\end{proposition}

\section{A General Scheme for Estimation}\label{gen}

In this section we construct a family of estimators $\hat u$ for $J$-minimising
solutions (cf. Definition \ref{defs:solution}) of Equation \eqref{intro:lineqn}
from noisy data $Y$ given by the white noise model   \eqref{intro:noisemodel}.
We define the estimators in a variational framework  and prove consistency as
well as convergence rates in terms of the Bregman-divergence w.r.t. $J$. 

\subsection{MR-Statistic and
SMR-Estimation}\label{gen:multiressec}

We introduce a class of similarity measures in order to
determine whether the residuals $Y - K\hat u$ for a given estimator $\hat u \in
U$ resemble a white noise process or not. To this end we will consider the
extreme-value distribution of projections of the residuals onto a predefined collection of
lines in $V$.  To this end, assume that
\begin{equation*}
  \Phi = \set{\phi_1,\phi_2,\ldots}\subset \overline{\ran(K)}\backslash\set{0}
\end{equation*}
is a fixed dictionary such that $\norm{\phi_n}\leq 1$ for all
$n\in\N$. For the sake of simplicity, we will frequently make use of the
abbreviation $\phi^*_n = \phi_n \slash \norm{\phi_n}$.

\begin{definition}\label{gen:mrstat}
Let $\set{t_N:\R^+\times (0,1] \ra \R}_{N\in¸\N}$ be a sequence of functions
that satisfy the following conditions
\begin{enumerate}[(i)]
  \item For all $r\in (0,1]$, the function $s\mapsto t_N(s,r)$ is convex,
  increasing and Lipschitz-continuous with Lipschitz-constants $L_{Nr}$ such
  that $ L_{Nr} \leq 
  L<\infty$ for all $N\in\N$ and
  \begin{equation}\label{gen:lowbnd}
    0 > \lambda_N(r):= \inf_{s\in \R^+} t_N(s,r) > -\infty.
  \end{equation}
  \item There exist constants $c_1, c_2 > 0$ and $\sigma_0 \in (0,1)$ such
  that for all $0<\sigma < \sigma_0$
  \begin{equation}\label{gen:ineq}
    t_N(s,r)\geq c_1 s + c_2 t_N(\sigma s, r)\quad\text{ for } (s,r)\in
    \R^+\times (0,1]\text{ and }N\in\N.
  \end{equation}
\end{enumerate}
Then, for $N\in\N$, the mapping $T_N: V \ra \R$ defined by
\begin{equation*}
  T_N(v) = \max_{1\leq n\leq N} t_N\left(\abs{\inner{v}{\phi_n^*}},
  \norm{\phi_n}\right)
\end{equation*}
is called a \emph{multiresolution statistic (MR-statistic)}.
\end{definition}

\begin{remark}\label{gen:remmean}
Let $\set{t_N}_{N\in\N}$ be a sequence of functions satisfying i) and ii) in
Definition \ref{gen:mrstat}. For a fixed $N\in\N$ the mappings $\mu_n:V\ra \R$
defined by
\begin{equation*}
\mu_n(v) = t_N(\abs{\inner{v}{\phi_n^*}}, \norm{\phi_n})
\end{equation*}
can be interpreted as the \emph{average} of the signal $v$ restricted to the
subspace spanned by $\phi_n^*$. With $\mu_n$ as above, the MR-statistic $T_N(v)$
in Definition \ref{gen:mrstat} takes the form \eqref{intro:mrstatrough} and
hence can be considered to measure the maximal local average of $v$ w.r.t. the
dictionary $\set{\phi_1,\ldots,\phi_N}$.
 \end{remark}

Definition \ref{gen:mrstat} allows for a vast class of MR-statistics  and the
conditions in (i) and (ii) appear rather technical. The
following example sheds some light on a special class of MR-statistics that
later on will be studied in more detail. We note, however, that our general
setting also applies to more involved statistics, as e.g. introduced in
\cite{frick:DueSpo01,frick:DueWal08}.

\begin{example}\label{gen:mrstatex}
  Assume that $\set{f_N:(0,1]\ra \R}_{N\in\N}$ is a sequence of positive
  functions and define
  \begin{equation*}
    t_N(s,r) := s - f_N(r).
  \end{equation*}
  Then, the assumptions in Definition \ref{gen:mrstat} are satisfied; to be
  more precise, we can set $L = 1$, $\lambda_N(r) = -f_N(r)$ and $c_1 =
  1-\sigma_0$ and $c_2 = 1$, where $\sigma_0\in(0,1)$ is arbitrary but fixed.
  Moreover, for a fixed $N\in\N$, the average functions $\mu_n:V\ra\R$ in Remark
  \ref{gen:remmean} read as
  \begin{equation*}
  \mu_n(v) = \abs{\inner{v}{\phi_n^*}} - f_N(\norm{\phi_n}).
  \end{equation*}
\end{example}

For a white noise process $\eps:V\ra \Ls{2}(\Omega, \mathfrak{A}, \Prob)$  and
$N\in \N$, consider the random variable
\begin{equation*}
  T_N(\eps) =  \max_{1\leq n\leq N} t_N\left(\abs{\eps(\phi_n^*)},
  \norm{\phi_n}\right).
\end{equation*}
Then, for a level $\alpha\in (0,1)$ we denote the
$(1-\alpha)$-quantile of $T_N(\eps)$ by $q_N(\alpha)$, that is,
\begin{equation}\label{gen:quantile}
  q_N(\alpha) := \inf\set{q\in\R~:~ \Prob\left( T_N(\eps)\leq q \right) \geq
  1 - \alpha}
\end{equation}

Our key paradigm is that an estimator $\hat u$ for a solution of
\eqref{intro:lineqn} fits the data $Y$ sufficiently well, if the statistic
$T_N(Y-K\hat u)$  does not exceed the threshold $q_N(\alpha)$ ($\alpha\in
(0,1)$ and $N\in\N$ fixed). Among all those estimators we shall pick the
\emph{most parsimonious} by minimising the functional $J$.

\begin{definition}\label{gen:smrest}
Let $N\in\N$ and $\alpha\in (0,1)$. Moreover, assume that $T_N$ is an
MR-statistic  and that $Y$ is given by \eqref{intro:noisemodel}.  Then every
element $\hat u_N(\alpha) \in U$ solving the convex optimisation problem
\eqref{intro:opt} is called a \emph{statistical multiresolution estimator
(SMRE)}.
\end{definition}

An SMRE $\hat u_N(\alpha)$ depends on the regularisation parameters $N\in\N$ and
$\alpha\in (0,1)$ that determine the admissible region $\mathcal{A}_N(\alpha)$
in \eqref{intro:admis}. In order to guarantee existence of a solution of the
convex problem in Definition \ref{gen:smrest}, that is existence of an SMRE,  it
is necessary to impose further standard assumptions:

\begin{assumption}\label{gen:assK}
There exists $N_0\in\N$ such that for all $c\in \R$ the sets
\begin{equation*}
  \Lambda(c) = \set{u\in U~:~ \max_{1\leq n\leq N_0}
  \abs{\inner{Ku}{\phi_n^*}} + J(u) \leq c }
\end{equation*}
are bounded in $U$.
\end{assumption}

Assumption \ref{gen:assK}  guarantees weak
compactness of the level sets of the objective functional $J$ restricted to the
admissible region $\mathcal{A}_N(\alpha)$. We note, that if $J$ is strongly
coercive (e.g. when $J$ is as in Example \ref{defs:norm}) then Assumption \ref{gen:assK} is satisfied without any
restrictions on the operator $K$. If $J$ lacks strong coercivity (as it is e.g. the case with the
total-variation semi-norm studied in Section \ref{tv}) additional properties of
$K$ are required in order to meet Assumption \ref{gen:assK}.

Application of standard arguments from convex optimisation yields

\begin{proposition}\label{gen:smrexist}
Assume that Assumption \ref{gen:assK} holds and let $N \geq
N_0$ and $\alpha\in (0,1]$. Then, an SMRE $\hat u_N(\alpha)$ exists.
\end{proposition}

Finally, we note that Assumption \ref{gen:assK} already implies the
requirements in Proposition \ref{defs:exjmin} and consequently existence of
$J$-minimising solutions.

\subsection{Consistency and Convergence Rates}\label{gen:conssec}

We investigate the asymptotic behaviour of $\hat u_N(\alpha)$ as the noise level
$\sigma$ in \eqref{intro:noisemodel} tends to zero.  According to the
reasoning following Definition \ref{gen:smrest}, the parameters $N\in\N$ and
$\alpha\in(0,1)$ can be interpreted as regularisation parameters and have to be
chosen accordingly: The model parameter $N$ has to be increased in order to
guarantee a sufficiently accurate approximation of the image space $V$, whereas
the test-level $\alpha$ tends to $0$ such that the true solution (asymptotically)
satisfies the constraints of \eqref{intro:opt} almost surely. We formulate
consistency and convergence rate results by means of the Bregman-divergence of
the SMRE $\hat u_N(\alpha)$ and a true solution $u^\dagger$ in terms of almost
sure convergence.

Throughout this section we shall assume  that $\set{\sigma_k}_{k\in\N}$ is a
sequence of positive noise-levels in \eqref{intro:noisemodel} such that
$\sigma_k \ra 0^+$ as $k\ra\infty$. Moreover, we assume that
$\set{\alpha_k}_{k\in\N} \subset(0,1)$ is a sequence of significance levels and
that $N_k\geq N_0$ is such that
\begin{equation}\label{gen:sumlim}
  \sum_{k=1}^\infty\alpha_k < \infty\quad\text{ and }\quad
  \lim_{k\ra\infty} N_k = \infty.
\end{equation}

\begin{theorem}\label{gen:consist} 
Assume that Assumptions \ref{defs:basicass} and \ref{gen:assK} hold. Let further
$u^\dagger$ be a $J$-minimising solution of \eqref{intro:lineqn} where $g\in \overline{\spa \Phi }$ and assume that 
\begin{equation*}
\sup_{N\in\N} T_N(\eps) < \infty
\end{equation*}
and 
\begin{equation}\label{gen:paramchoice}
  \zeta_k :=  \sigma_k \max\left(\inf_{1\leq n\leq N_k}
  \lambda_{N_k}(\norm{\phi_n}), \sqrt{-\log \alpha_k}\right)\ra 0.
\end{equation}
Then, for $\hat u_k := \hat u_{N_k}(\alpha_k)$ as in \eqref{intro:opt} one
has
\begin{equation}\label{gen:conseqn}
  \sup_{k\in\N}\norm{\hat u_k} < \infty,\quad J(\hat u_k) \ra
  J(u^\dagger)\quad\text{ and }\quad D_J(u^\dagger, \hat u_k)  \ra 0\quad\text{ a.s. }
\end{equation}
as well as
\begin{equation}\label{gen:imgconseqn}
  \limsup_{k\ra\infty} \max_{1\leq n\leq N_k} \frac{ 
  \abs{\inner{\phi_n^*}{K\hat u_k - K u^\dagger}}}{\zeta_k} < \infty
  \quad
  \text{ a.s.} 
\end{equation}
\end{theorem}

Theorem \ref{gen:consist} states that if for a given vanishing sequence of
noise levels $\sigma_k$, suitable (in the sense of
\eqref{gen:paramchoice}) sequences of regularisation parameters $N_k$ and
$\alpha_k$ can be constructed, then the sequences of corresponding SMRE
converges to a true $J$-minimising solution $u^\dagger$ w.r.t. the
Bregman-divergence. We note that the assumption on the boundedness
of MR-statistic $T_N(\eps)$ is crucial and in general
non-trivial to show.

It is well known that without further regularity
restrictions on $u^\dagger$, the speed of convergence in \eqref{gen:conseqn} can
be arbitrarily slow. \emph{Source conditions} as in Definition
\ref{defs:solution} (iii) are known to constitute sufficient regularity
conditions with quadratic fidelity  (cf.
\cite{frick:BisHohMunRuy07,frick:LouLud08,frick:MaiRuy96}). In our situation,
where the fidelity controls the maximum over all residuals, we additionally have to assume
that the source elements exhibit certain approximation properties:

\begin{assumption}\label{gen:assSE}
There exists a $J$-minimising solution $u^\dagger$ of \eqref{intro:lineqn} that
satisfies the source condition \eqref{defs:solution:source}  with source
element $p^\dagger$. Moreover, for $n,N\in\N$ there exist  $b_{n,N}\in \R$
such that
\begin{equation}\label{gen:seapprox}
  \err_N(p^\dagger):= \norm{p^\dagger - \sum_{n=1}^N b_{n,N} \phi_n^*}\ra0
  \quad\text{ and } \quad\sup_{N\in\N} \sum_{n=1}^N
  \abs{b_{n,N}} < \infty.
\end{equation}
\end{assumption}

\begin{remark}\label{gen:assSErem}
\begin{enumerate}[i)]
  \item Assumption \ref{gen:assSE} amounts to say that there exists a
  $J$-minimising solution $u^\dagger$ that satisfies the source condition
  \eqref{defs:solution:source} with a source element $p^\dagger$ that can be
  approximated sufficiently well by the dictionary $\Phi$ in use.
  From \eqref{defs:solution:source} it becomes clear that we can always assume
  that $p^\dagger \in \overline{\ran(K)}$, such that the first condition in
  \eqref{gen:seapprox} is not very restrictive, in fact.
  
  \item Good estimates of approximation errors for non-orthogonal dictionaries
  $\Phi$ are hard to come up with in general. Examples of non-orthogonal
  dictionaries where such estimates are available are wavelet-
 \citep{frick:Dau92} and curvelet- \citep{frick:CanDon04} frames.
    
  \item It is important to note that, given prior information on the true
  solution $u^\dagger$, the conditions in Assumption \ref{gen:assSE} may indicate
  whether a given dictionary is well suited for the
  reconstruction of $u^\dagger$ or not. As we will see in Section \ref{simple},
  a priori information on the smoothness of $u^\dagger$ can typically be
  employed. 
\end{enumerate}
\end{remark}

\begin{theorem}\label{gen:thmrates}
Let the requirements of Theorem \ref{gen:consist} be satisfied and assume
further that Assumption \ref{gen:assSE} holds with $g\in \overline{\spa \Phi
}$. If $\eta_k := \max(\zeta_k, \err_{N_k}(p^\dagger)) \ra 0$, then 
\begin{equation}\label{gen:rateeqn}
  \limsup_{k\ra\infty} \frac{D_J^{K^*p^\dagger}(\hat u_k,
  u^\dagger)}{ \eta_k} < \infty \quad \text{ and }\quad \limsup_{k\ra\infty}
  \max_{1\leq n\leq N_k} \frac{ \abs{\inner{\phi_n^*}{K\hat u_k - K
  u^\dagger}}}{ \eta_k} < \infty \quad \text{ a.s.} 
  \end{equation}
\end{theorem}
  
\begin{remark}\label{gen:ratesrem}
The convergence rate result in Theorem \ref{gen:thmrates} is rather general, in
the sense that the rate function $\eta_k$ in \eqref{gen:rateeqn} has to be
determined for each choice of $K$, $J$ and $\Phi$ separately. We outline a
general procedure how this can be done in practice: Assume that $u^\dagger$ is a
$J$-minimising solution of \eqref{intro:lineqn} that satisfies Assumption
\ref{gen:assSE} with a source element $p^\dagger$.  
\begin{enumerate}[(i)]
  
  \item The sequence $\set{-\inf_{1\leq n\leq N} \lambda_N(\norm{\phi_n})
  }_{N\in\N}$ is positive according to \eqref{gen:lowbnd}. Hence
  \begin{equation*}
    N_k:= \inf\set{ N\in\N~:~ \err_N(p^\dagger) \leq- \sigma_k \inf_{1\leq n\leq
    N} \lambda_N(\norm{\phi_n}) }
  \end{equation*}
  is well-defined and since $\set{\sigma_k}_{k\in\N}$ is
  non-increasing one has $N_k \leq N_{k+1}$ and $N_k \ra \infty$ as
  $k\ra\infty$.
  \item After setting $\eta_k =-\sigma_k \inf_{1\leq n\leq N_k}
  \lambda_{N_k}(\norm{\phi_n})$ it remains to check that the sequence of
  test-levels $\alpha_k = \exp\left(-\left(\kappa
  \eta_k\slash \sigma_k \right)^2\right)$ is summable (for some constant
  $\kappa > 0$).
\end{enumerate}
For the so constructed sequences $N_k$, $\eta_k$ and $\alpha_k$, the
assertions of Theorem  \ref{gen:thmrates} hold. 
\end{remark}
 
As we will see in Section \ref{simple}, the procedure in Remark
\ref{gen:ratesrem} typically results in convergence rates $\eta \sim \sigma
\sqrt{-\log \sigma}$. For orthogonal dictionaries $\Phi$ it will turn out in
Section \ref{simple:indep} that these rates are nearly optimal for the
smoothness class induced by Assumption \ref{gen:assSE} (cf. Example
\ref{simple:svd} below). It is an open question what the optimal rates are for
general (non-orthogonal) dictionaries.

\section{Applications and Examples}\label{simple}

In Section \ref{gen} we developed a general method for estimation of
$J$-minimising solutions of linear and ill-posed operator equations from  noisy
data. Our estimation scheme thereby employes the MR-statistic
$T_N$ (cf. Definition \ref{gen:mrstat}).  In this section we will study
particular instances of MR-statistics  covered by the general theory in Section
\ref{gen}:
\begin{itemize}
  \item We study the case where $T_N$ constitutes the extreme-value statistic of
the coefficients w.r.t. an orthonormal dictionary $\Phi$ (Section
\ref{simple:indep}). We show how Assumption \ref{gen:assSE} in this case reduces
to the requirement that the true solution $u^\dagger$ lies in a
Sobolev-ellipsoid w.r.t. the system $\Phi$. Moreover, it will turn out that for
the case when $\Phi$ denotes the eigensystem of a compact operator, SMRE can be
considered as soft-thresholding.
  \item In Section \ref{lutz} we skip the assumption of orthonormality and  examine
general SMREs w.r.t. (non-orthonormal)  dictionaries that
satisfy certain entropy conditions. In particular, we will consider the case
when $U = V = \Ls{2}([0,1]^d)$ and when $\Phi$ consists of indicator
functions w.r.t. a redundant system of subcubes in $[0,1]^d$.
\item Finally, we study the case when the penalty functional $J$
is chosen to be the total-variation semi-norm on $U = \L{2}$ in Section
\ref{tv}.  We highlight the implications of our general
convergence rate results for image deconvolution and complement the theoretical
results by some numerical examples. In particular, we compare our approach to
the locally adaptive image reconstruction method recently introduced in
\cite{frick:Gra09}.
\end{itemize}

Throughout this section we assume that Assumptions \ref{defs:basicass} and
\ref{gen:assK} hold. Moreover, we shall agree upon
$\set{\sigma_k}_{k\in\N}$ being a sequence of noise levels such that $\sigma_k
\ra 0^+$ and that for $k\in\N$ there are $\alpha_k \in (0,1)$ and
$N_k\in\set{N_0, N_0+1,\ldots}$ such that \eqref{gen:sumlim} holds.

\subsection{Introductory Example: Gaussian Sequence Model}\label{simple:indep}
In this section we shall consider the case where the dictionary $\Phi  =
\set{\phi_1,\phi_2,\ldots}$ constitutes an orthonormal basis of
$\overline{\ran(K)}$. Evaluation of Equation \eqref{intro:noisemodel} at the
elements $\phi_n$ hence yields
\begin{equation*}
  y_n =  \theta_n + \sigma \eps_n,
\end{equation*}
where $Y(\phi_n) = y_n$, $\theta_n = \inner{Ku}{\phi_n}$ and $\eps_n =
\eps(\phi_n)$. We define the MR-statistic $T_N$ by setting
$t_N(s,r) = s - \sqrt{2\log N}$ in Definition \ref{gen:mrstat}. In
other words, we consider the maximum of the coefficients w.r.t to the dictionary
$\Phi$, that is
\begin{equation}\label{simple:simplestat}
  T_N(v) = \max_{1\leq n\leq N}\abs{\inner{v}{\phi_n}} - \sqrt{2\log N}.
\end{equation}
Since $\set{\phi_1,\phi_2,\ldots}$ are linearly independent and
normalised, it follows that the random variables $\eps_1,\eps_2,\ldots$ are
independent and standard normally distributed. This implies that $
T_N(\eps)$ is bounded almost surely.
% since $T_N(\eps)$ is increasing and
% $\sup_{N\in\N} T_N(\eps) = \lim_{N\ra\infty}\left(\max_{1\leq n\leq N}
% \abs{\eps_n} - \sqrt{2\log N} \right) = 0$ a.s.

In what follows, we will apply Theorems \ref{gen:consist} and
\ref{gen:thmrates} to the present case. To this end, we observe that for
$\sigma > 0$ and $ N\in\N$ 
\begin{equation*}
  -\sigma \inf_{1\leq n\leq N} \lambda_{N} (\norm{\phi_n}) =
  \sigma \sqrt{2\log N}.
\end{equation*}
With the above preparations, we are able to reformulate the consistency result
in Theorem \ref{gen:consist}.

\begin{corollary}\label{simple:conscorindep}
Let $u^\dagger \in U$ be a $J$-minimising solution
of \eqref{intro:lineqn} where $g\in \overline{\spa \Phi}$. Moreover, assume
that $\sigma_k^2 \max(\log N_k,-\log \alpha_k) \ra 0$.
Then, the SMRE $\hat u_k = \hat u_{N_k}(\alpha_k)$ almost surely satisfies
\eqref{gen:conseqn} and \eqref{gen:imgconseqn}.
\end{corollary}

In order to apply the convergence rate result in Theorem \ref{gen:thmrates},
Assumption \ref{gen:assSE} has to be verified. We set $b_{n,N}\equiv
\inner{p^\dagger}{\phi_n}$ in Assumption \ref{gen:assSE}. Note that   the
expression $\err_N(p)$ denotes the  approximation error of the $N$-th partial
Fourier-series w.r.t. $\Phi$. Thus, Assumption \ref{gen:assSE} is linked to
absolute summability of the Fourier-coefficients w.r.t. the basis $\Phi$, i.e.
\begin{equation}\label{simple:sourcerep}
  \sum_{n=1}^\infty \abs{\inner{p^\dagger}{\phi_n}} < \infty
\end{equation}
The \emph{Bernstein-Stechkin criterion} is a classical method for
testing for absolute summability. We present a version suitable for our purpose
in the following

\begin{proposition}\label{simple:berstecrit}
Let $p^\dagger \in V$. Then, \eqref{simple:sourcerep} is satisfied if $\sum_{N =
1}^\infty \err_N(p^\dagger)\slash \sqrt{N}  < \infty$.
\end{proposition}

\begin{proof}
  The classical version of the Bernstein-Stechkin Theorem \citep[see e.g. ][Thm.
  7.4]{frick:Lau73} states that for each $f\in \Ls{2}([0,1])$ and each ON-basis
  $\underline{v} = \set{v_1,v_2,\ldots}$ of $\Ls{2}([0,1])$, the
  Fourier-coefficients of $f$ are absolutely summable, if $\sum_{N =
1}^\infty \err_N(p^\dagger)\slash \sqrt{N}  < \infty$. Since
  each separable Hilbert space is isometrically isomorphic to $\Ls{2}([0,1])$,
  the assertion finally follows.
\end{proof}
 
Following the procedure outlined in Remark \ref{gen:ratesrem} (Section
\ref{gen}) we define
\begin{equation}\label{simple:sequences}
  N_k := \inf\set{N\in\N~:~\err_N(p^\dagger) \leq \sigma_k\sqrt{2\log
  N}}\quad\text{ and }\quad \eta_k := \sigma_k \sqrt{2\log N_k}.
\end{equation}

\begin{corollary}\label{simple:ratescordep}
Let $g\in V$ be attainable and $u^\dagger \in U$ be a $J$-minimising solution of
\eqref{intro:lineqn} that satisfies the source condition with a source element
$p^\dagger$ such that the condition in Proposition \ref{simple:berstecrit}
holds. Moreover, let $N_k$ and $\eta_k$ be defined  as in
\eqref{simple:sequences}. If
\begin{equation*}
  \alpha_k := e^{-\left(\frac{\kappa \eta_k}{\sigma_k}\right)^2} =
  N_k^{-2\kappa^2} \in \ell^1(0,1)
\end{equation*}
for a constant $\kappa > 0$, then the SMRE $\hat u_k = \hat
u_{N_k}(\alpha_k)$ almost surely satisfies \eqref{gen:rateeqn}.
\end{corollary}

The problem of characterising those elements $p^\dagger\in V$ that satisfy the
assumption of Proposition \ref{simple:berstecrit} is a classical issue in
Fourier-analysis and approximation theory. Sufficient condition are usually
formalised by characterising the decay properties of the Fourier-coefficients.
In a function space setting, this leads to particular smoothness classes of
functions and in the general situation can be given in terms of \emph{Sobolev
ellipsoids}. For constants $\beta, Q > 0$ we define $\Theta(\beta, Q)$ as the
infinite-dimensional ellipsoid
\begin{equation}\label{simple:sobell}
  \Theta(\beta,Q) = \set{ \theta \in \ell^2~:~ \sum_{n\in\N} n^{2\beta}
  \theta_n^2 \leq Q^2}.
\end{equation}
The \emph{Sobolev class} $W(\beta, Q)\subset V$ is then defined to consists of
all $v\in V$ such that $\set{\inner{v}{\phi_n}}_{n\in\N}\subset \Theta(\beta,Q)$ 
\citep[see][Sec.1.10.1]{frick:Tsy09}. For $v\in W(\beta,Q)$ we have that
Proposition  \ref{simple:berstecrit} is applicable if $\beta > 1\slash 2$.

\begin{example}\label{simple:svd}
  Assume that $J(u) = \frac{1}{2}\norm{u}^2$ and let $K$ be a compact operator
  with singular value decomposition (SVD) $\set{(\psi_n,\phi_n,s_n)}_{n\in\N}$:
  $\set{\psi_n}_{n\in\N}$ is an orthonormal basis (ONB) of $\ker(K)^\bot$,
  $\set{\phi_n}_{n\in\N}$ is an ONB of $\overline{\ran(K)}$ and the singular values
  $\set{s_n}_{n\in\N}$ are positive and $s_n\ra 0$ as $n\ra\infty$. Moreover
  \begin{equation}\label{simple:svdeqn}
    K\psi_n = s_n \phi_n\quad\text{ and }\quad K^* \phi_n = s_n \psi_n,
  \end{equation}
  for all $n\in\N$. For $N\in\N$ and $\alpha\in (0,1]$ it turns out (e.g. by
  applying the method of Lagrangian multipliers) that the SMRE $\hat
  u_N(\alpha)$ with $T_N$ as in \eqref{simple:simplestat} is a \emph{shrinkage
  estimator} given by
  \begin{equation*}
    \hat u_N(\alpha) = \sum_{n=1}^N s_n^{-1} y_n \left(1 -
    \frac{q_N(\alpha) + \sqrt{2\log N}}{\abs{y_n}}\right)_+ \psi_n.
  \end{equation*}
  We note that $\hat u_N(\alpha)$ is a particular instance of a soft
  thresholding estimator.
  
  Now, let $u^\dagger\in U$ be a minimum-norm solution of \eqref{intro:lineqn}
  that satisfies the source condition $K^*p^\dagger = u^\dagger$
  (cf. Example \ref{defs:norm})
  with source element $p^\dagger\in W(\beta,Q)$ for $Q > 0$ and $\beta > 1\slash 2$. Then,
  $\err_N(p^\dagger) \leq Q N^{-\beta}$ and it follows from \eqref{simple:sequences} that
  \begin{equation*}
    N_k \sim \left(\frac{Q}{\sigma_k}\right)^2\quad\text{
    and }\quad \eta_k \sim \sigma_k\sqrt{-\log \sigma_k}.
  \end{equation*}
  If $\sigma_k$ has polynomial decay, we can choose a constant $\kappa > 0$
  such that $\alpha_k = \exp(-(\kappa \eta_k \slash \sigma_k)^2) =
  \sigma_k^{\kappa^2}$ is summable and it follows from Corollary
  \ref{simple:ratescordep} and Example \ref{defs:norm} that
  \begin{equation*}
    \limsup_{k\ra\infty} \frac{1}{\sigma_k\sqrt{-\log
    \sigma_k}}\norm{u^\dagger - \hat u_{N_k}(\alpha_k)}^2 < \infty\quad\text{
    a.s.}
  \end{equation*}
  If the operator equation $Ku=g$ is \emph{mildly ill-posed}, i.e. $s_n\sim
  n^{-\gamma}$ for some $\gamma > 0$, then the equation $K^*p^\dagger =
  u^\dagger$ together with $p^\dagger \in W(\beta, Q)$ implies that $u^\dagger
  \in W(\beta + \gamma, Q)$. The optimal rates (w.r.t. the quadratic risk) are
  known to be of order $\psi(\beta) = \sigma_k^{4(\beta+\gamma)\slash (4\lambda
  + 2\beta + 1)}$ (cf. \cite[Thm. 1]{frick:Cav08}). Since $\psi(\beta) \to
  \sigma_k$ for $\beta\to 1\slash 2$ the convergence rate implied by Theorem
  \ref{gen:thmrates} is optimal (up to a $\log$-factor). 
\end{example}

As mentioned above, sufficient conditions for the Bernstein-Stechkin
criterion (cf. Proposition \ref{simple:berstecrit}) in a function space
setting are usually formalised in characterising smoothness properties. The following example shows
how this applies to H\"older-continuity.

\begin{example}\label{simple:trigo}
  Let $V = \Ls{2}_{\text{per}}([0,1])$ be the Hilbert space of all
  square-integrable and periodic functions on the unit interval. Moreover,  we
  assume that $\overline{\ran(K)} =
  \Ls{2}([0,1])$ and consider the \emph{trigonometric basis}
  \begin{equation*}
    \phi_{2n} = \sqrt{2}\cos(n\pi x)\quad\text{ and }\quad \phi_{2n+1}
    = \sqrt{2}\sin(n\pi x).
  \end{equation*}
  Assume that  $p^\dagger \in \mathcal{H}_\beta([0,1])\cap V$ (cf. Definition
  \ref{app:hoelder}) with $\beta > 1\slash 2$. Then we have that
  $\err_N(p^\dagger) \leq Q N^{-\beta}\log N$ for a suitable constant $Q>0$ and
  therefore it follows from  Proposition \ref{simple:berstecrit} that
  \eqref{simple:sourcerep} holds.
  
  Hence, if $u^\dagger$ is a $J$-minimising solution of \eqref{intro:lineqn}
  that satisfies the source condition \eqref{defs:solution:source} with source
  element $p^\dagger \in \mathcal{H}_\beta([0,1])$  and
  if the sequences $N_k, \eta_k$ and $\alpha_k$ are chosen as in Example
  \ref{simple:svd}, then $\hat u_k = \hat u_{N_k}(\alpha_k)$  almost surely
  satisfy \eqref{gen:rateeqn}.
\end{example}

\begin{remark}
\begin{enumerate}[i)]
  \item The assertions of Example
  \ref{simple:trigo} still hold if the trigonometric basis is replaced by any
  other  orthonormal basis $\set{\phi_n}_{n\in\N}$ of $\overline{\ran(K)}$ such
  that the Bernstein-Stechkin criterion in Proposition \ref{simple:berstecrit}
  is satisfied. This holds for example for a vast class of orthonormal wavelet bases of $\Ls{2}([0,1])$ as studied in
  \cite{frick:CohDauVia93}.
  \item For the trigonometric basis in Example \ref{simple:trigo}, the
  Bernstein-Stechkin criterion \ref{simple:berstecrit} can be replaced by the
  requirement that $p^\dagger\in \mathcal{H}_\beta([0,1])$ for any $\beta>0$
  is additionally of bounded variation \citep[see][Vol.1\
  Thm.3.6]{frick:Zyg77}.
\end{enumerate}
\end{remark}

\subsection{Non-orthogonal Models}\label{lutz}

In contrast to Section \ref{simple:indep}, where we considered orthonormal
dictionaries, we will now focus on more general (non-orthonormal)
systems. In other words,  we consider sequences
\begin{equation*}
  \Phi = \set{\phi_1,\phi_2, \dots } \subset \overline{\ran(K)}\backslash\set{0}
\end{equation*}
and assume that $\norm{\phi_n}\leq 1$ for all $n\in\N$. Moreover, we will
make use of the MR-statistic $T_N$ (cf. Definition \ref{gen:mrstat})
defined by
\begin{equation}\label{lutz:statistic}
  t_N(s,r) = s - \sqrt{-2\gamma \log r},\quad (s,r)\in \R^+\times (0,1]
\end{equation}
where $\gamma > 0$ is some constant. As outlined in Example \ref{gen:mrstatex},
one verifies that $t_N(s,r)$ satisfies the assumptions of Definition
\ref{gen:mrstat}. In particular, we find that $\lambda_N(r) = - \sqrt{-2\gamma
\log r} > -\infty$ for all $r\in (0,1]$. The parameter $\gamma$ that appears in
\eqref{lutz:statistic} has to be chosen  appropriately in dependence on $\Phi$
in order to guarantee that the MR-statistic  $T_N(\eps)$ is bounded almost
surely. A sufficient condition on $\gamma$ has  for example been given in
\cite[Thm 7.1]{frick:DueWal08}

\begin{proposition}\label{lutz:teststatfinite}
If there exists constants $A,B>0$ such that
\begin{equation}\label{mrnorm:caprequ}
  D(u \delta, \set{\phi \in \Phi~:~ \norm{\phi}\leq \delta}) \leq A u^{-B}
  \delta^{-\gamma},\quad \text{ for all } u,\delta \in (0,1]
\end{equation}
then almost surely $\sup_{N\in\N} T_N(\eps) < \infty$. Here $D$  denotes the
capacity number (cf. Definition  \ref{lutz:capcovdef}).
\end{proposition}

\begin{corollary}\label{lutz:consistcor}
Let $u^\dagger \in U$ be a $J$-minimising solution of
\eqref{intro:lineqn} where $g\in \overline{\spa{\Phi}}$ and $\gamma>0$ be
chosen such that the assumption of Proposition \ref{lutz:teststatfinite} is
satisfied. Moreover, assume that 
\begin{equation*}
\sigma_k^2 \min( \min_{1\leq n\leq N_k}
\log\left(\norm{\phi_n}\right), \log \alpha_k) \ra 0.
\end{equation*}
 Then, the SMRE $\hat u_k
= \hat u_k(\alpha_k)$ almost surely satisfies \eqref{gen:conseqn}.
\end{corollary}

In order to apply the convergence rate results in Theorem \ref{gen:thmrates}, it
is necessary that a $J$-minimising solution $u^\dagger$ of \eqref{intro:lineqn}
satisfies the source condition \eqref{defs:solution:source} with a source element
$p^\dagger$ that can be approximated by the dictionary $\Phi$
sufficiently well (cf. Assumption \ref{gen:assSE}). We illustrate the
assertion of Theorem \ref{gen:thmrates} when  $U = V = \Ls{2}([0,1]^d)$ ($d\geq
1$) and when $\Phi$ consists of a countable selection of indicator functions
on cubes in $[0,1]^d$ (cf. Example \ref{intro:ex}). 

First, we shall examine when Proposition
\ref{lutz:teststatfinite} holds. To this end, we will focus first on the
(uncountable) collection $\Phi_s$ of indicator functions on cubes in $[0,1]^d$.
Then, according to Proposition \ref{mrnorm:allsetass}, the assumptions of
Proposition \ref{lutz:teststatfinite} are satisfied for $\Phi = \Phi_\text{s}$
and  $\gamma = d$. Particularly, it follows that the assertion of
Proposition \ref{lutz:teststatfinite} also holds for arbitrary (countable)
sub-systems $\Phi\subset\Phi_s$, that is the statistic
\begin{equation}\label{lutz:mrstat}
  T_N(\eps) = \max_{1\leq n\leq N} \abs{\eps(\chi_{Q_n})} - \sqrt{-d
  \log\abs{Q_n} }\quad\text{ where }\quad\chi_{Q_n}\in \Phi
\end{equation}
stays bounded a.s. as $N\ra\infty$ (note here, that $\norm{\chi_{Q_n}} =
\sqrt{\abs{Q}}$).

Next, we study Assumption \ref{gen:assSE} in the present setting. Let
$\mathcal{P} = \set{Q_1, Q_2, \ldots}$ be a countable
system of cubes and set $\Phi = \set{\chi_{Q_n}~:~ n\in\N}$. We shall assume
that $\mathcal{P}$ satisfies the conditions of Lemma \ref{reg:approx} (where
$\Omega = [0,1]^d$ and $A_i = Q_i$ for $i\in\N$). Let $\set{n_l}_{l\in\N}$ and
$\set{\delta_l}_{l\in\N}$ be defined accordingly. Moreover, we define
\begin{equation*}
  \eps_l = \inf_{n_l<j\leq n_{l+1}} \sqrt{\abs{Q}} = \inf_{n_l<j\leq
  n_{l+1}} \norm{\chi_{Q_j}},
\end{equation*}
where we assume that $\set{\eps_l}_{l\in\N}$ is non-increasing. This means that we
partition the set $[0,1]^d$ into disjoint sub-cubes
$\set{Q_j}_{n_l<j\leq n_{l+1}}$ whose size (or \emph{scale}) is bounded by
$[\eps_l, \delta_l]$. It is more natural to formulate convergence rate results in
terms of the total number $m$ of used scales rather than in the total number of
sub-cubes $N = N(m) = n_{m+1}$. Following  Remark
\ref{gen:ratesrem} and applying Lemma \ref{reg:approx} we therefore define
for a given continuous function $p^\dagger:[0,1]^d\ra \R$
\begin{equation}\label{reg:sequences}
  m_k := \inf\set{ m\in\N~:~ \frac{m+1}{\sum_{\nu = 0}^m
  \omega^{-2}(\delta_\nu, p^\dagger)} \leq -2\sigma_k^2 \log \eps_m
  }\;\text{ and } \; \eta_k := \sigma_k \sqrt{-2\log \eps_{m_k}}.
\end{equation}
Here $\omega(\cdot, p^\dagger)$ denotes the modulus of continuity of
$p^\dagger$ (cf. Definition \ref{app:hoelder}). With this and the general
convergence rate result in Theorem \ref{gen:thmrates} we  obtain

\begin{corollary}\label{reg:convrate}
Let $u^\dagger\in\Ls{2}([0,1]^d)$ be a $J$-minimising solution of
\eqref{intro:lineqn} where $g \in \overline{\spa{\Phi}}$ and that satisfies the
source condition \eqref{defs:solution:source} with source element $p^\dagger
\in C([0,1]^d)$.  Moreover, let $m_k$ and
$\eta_k$ be defined as in \eqref{reg:sequences}. If
\begin{equation*}
  \lim_{k\ra\infty} \eta_k = 0\quad\text{ and }\quad \alpha_k :=
  e^{-\left(\frac{\kappa \eta_k}{\sigma_k}\right)^2} = \eps_{m_k}^{-2\kappa^2}
  \in \ell^1(0,1)
\end{equation*}
for a constant $\kappa > 0$, then the SMRE $\hat u_k = \hat
u_{N(m_k)}(\alpha_k)$ almost surely satisfy \eqref{gen:rateeqn}.
\end{corollary}

\begin{example}\label{lutz:dyadicpart} We consider the system of all dyadic
  partitions $\mathcal{P} = \mathcal{P}_2$ of $[0,1]^d$ as in Example
  \ref{mrnorm:dyadicsetass}. In particular, we note that the assumptions of Lemma
  \ref{reg:approx} are fulfilled with $n_l = (2^{d(l+1)} - 1)\slash (2^d-1)$,
  $\delta_l = 2^{-l} \sqrt{d}$ and $\eps_l = 2^{-ld\slash 2}$.
  
  If $p^\dagger\in \mathcal{H}_\beta([0,1]^d)$ for $0<\beta\leq 1$, then there
  exists a constant $Q = Q(p^\dagger) >0$ such that $\omega(\delta_l,
  p^\dagger)\leq Q \delta_l^\beta$. This shows that
  \begin{equation*}
    \frac{m+1}{\sum_{\nu = 0}^m
    \omega^{-2}(\delta_\nu, p^\dagger)} \leq Q^2 d^\beta(2^{2\beta} - 1 )
    \frac{m+1}{2^{2\beta (m+1)}-1}
  \end{equation*}
  for $m\in\N$ large enough. From this and \eqref{reg:sequences} it is easy to
  see, that
  \begin{equation*}
    m_k +1 \sim \frac{1}{2\beta \log 2}\log\left(\frac{Q^2 d^\beta (2^{2\beta}
    -1)}{d \sigma_k^2 \log 2  } +1 \right)\;\text{ and }\; \eta_k \sim
    \sigma_k\sqrt{-\log \sigma_k}.
  \end{equation*}
  Thus, if there exists a constant $\kappa > 0$ such that
  \begin{equation*}
    \alpha_k = e^{-\left(\frac{\kappa \eta_k}{\sigma_k}\right)^2} =
    \sigma_k^{\kappa^2}
  \end{equation*}
  is summable and if the true  $J$-minimising solution  $u^\dagger$ satisfies
  the source condition \eqref{defs:solution:source} with  source element
  $p^\dagger \in \mathcal{H}_\beta([0,1])$, then it follows that the
  SMRE $\hat u_k = \hat u_{N(m_k)}(\alpha_k)$  almost surely satisfy
  \eqref{gen:rateeqn} with $\eta_k=\sigma_k\sqrt{-\log \sigma_k}$.
\end{example}

\subsection{TV-Regularisation for Imaging}\label{tv}

In this section we will study the theoretical properties of SMRE for the special
case where $J$ denotes the \emph{total-variation semi-norm} of measurable,
bi-variate functions. It has been argued (e.g. in \cite{frick:RudOshFat92}) that
this has a particular appeal for linear inverse problems arising in imaging (such as
deconvolution), since discontinuities along curves (edges, that is) are not
smoothed by minimising $J$.

% Over the last years regularisation of (inverse) regression problems in a single
% space dimension invoking the total-variation semi-norm has been studied
% intensively and efficient numerical methods, such as the \emph{taut-string
% algorithm} in  \cite{frick:DK01}, have been proposed (see e.g.
% \cite{frick:DK01, frick:DavKovMei09, frick:MamGee97} and references therein). In two or more space
% dimensions, however, the situation is much more involved and a generalisation
% is difficult \citep[see e.g.][]{frick:HinHinKunOehSch03}.  We study here an
% extension to the case of space dimension $2$ as well as to deconvolution by applying the results in
% Section \ref{gen} to the following setting:

We assume henceforth that $\Omega\subset \R^2$ is an open and bounded
domain with Lipschitz-boundary $\partial \Omega$ and outer unit normal $\nu$.
Moreover,  we set $U = \L{2}$ and define $\BV$ to be the collection of $u\in U$
whose derivative $\D u$ (in the sense of distributions) is a signed
$\R^2$-valued Radon-measure with finite total-variation $\abs{\D u}$, that is
\begin{equation*}
  \abs{\D u}(\Omega)= \sup_{\substack{\psi \in C_0^1(\Omega, \R^2) \\
  \abs{\psi} \leq 1}} \int_\Omega \diw{\psi} u \dx <\infty.
\end{equation*}
We note that the norm $\norm{u}_{\TV}:= \norm{u}_{\Ls{1}} + \abs{\D u}(\Omega)$
turns $\BV$ into a Banach-space and that with this norm $\BV$ is continuously
embedded into $\L{2}$. The embedding is even compact if $\L{2}$ is replaced by
$\L{p}$ with $p < 2$ (a proof of these embedding results can be found in
\cite[Thm. 2.5]{frick:AV94}. For an exhaustive treatment of $\BV$ see
\cite{frick:Z89}). With this, we define
\begin{equation}\label{tv:tv}
  J(u) = \begin{cases}
\abs{\D u}(\Omega) & \text{ if }u\in\BV \\
+\infty & \text{ else.}
\end{cases}
\end{equation}
The functional $J$ is convex and proper and, as it was shown e.g. in
\cite[Thm. 2.3]{frick:AV94}, $J$ is lower semi-continuous on $\L{2}$. This
shows, that $J$ satisfies Assumption \ref{defs:basicass} (ii). Next, we examine
Assumption \ref{gen:assK}:

\begin{lemma}\label{tv:compact}
If there exists $n_0\in \N$ such that $\abs{\inner{K\mathbf{1}}{\phi_{n_0}}} >
0$ then Assumption \ref{gen:assK} holds. Here, $\mathbf{1}$ denotes the constant
$1$-function on $\Omega$.
\end{lemma}

\begin{proof}
 Let $c\in \R$ and $\set{u_k}_{k\in\N}\subset \Lambda(c)$. Then in
  particular it follows that $\sup_{k\in\N} J(u_{k_n}) \leq c < \infty$  and thus we find with
  Poincar\'e's inequality \citep[see][Thm. 5.11.1]{frick:Z89}
  \begin{equation*}
    \norm{u_k - \bar u_k}_{\Ls{2}} \leq c_1 J(u_k) \leq c_2 < \infty
  \end{equation*}
  for suitable constants $c_1, c_2 \in\R$, where $\bar u_k =
  \abs{\Omega}^{-1} \int_\Omega u_k(\tau) \diff \tau $. Now choose $\phi
  \in \set{\phi_1,\dots,\phi_N}$ and observe that
  \begin{multline*}
    \frac{\abs{\bar u_k} \abs{\inner{\phi}{K\mathbf{1}}}}{\norm{\phi}} =
    \frac{\abs{\inner{\phi}{K \bar u_k}}}{\norm{\phi}}  \leq \frac{\abs{\inner{
    \phi }{K(\bar u_k - u_k)}}}{\norm{\phi}} + \frac{\abs{\inner{\phi}{K u_k}}}{\norm{\phi}} \\
    \leq \norm{K} \norm{u_k - \bar u_k}_{\Ls{2}} + \max_{1\leq n\leq N}
    \frac{\abs{\inner{Ku_k}{\phi_n}}}{\norm{\phi_n}}\leq \norm{K} c_2 + c.
  \end{multline*}
  Let $1\leq n_0\leq N$ be such that
  $\abs{\inner{K\mathbf{1}}{\phi_{n_0}}} =: \gamma > 0$. Then,
  $\abs{\bar u_n} \leq (\norm{K} c_2 + c)\norm{\phi_{n_0}} \slash \gamma =:c_3$
  and we find 
  \begin{equation*}
    \norm{u_n}_{\Ls{2}}  \leq \left(\norm{u_n
    - \bar u_n}_{\Ls{2}} + \norm{ \bar u_n}_{\Ls{2}} \right)
    \leq c_2 + c_3 \abs{\Omega}.
  \end{equation*}.
\end{proof}

We note that the assumptions in Lemma \ref{tv:compact} already imply the weak
compactness of the sets \eqref{defs:comset} and thus guarantee existence of a
$J$-minimising solution of \eqref{intro:lineqn}. From the above cited embedding
properties of the space $\BV$ it is easy to derive an improved version of the
consistency result in Theorem \ref{gen:consist}.
 
\begin{corollary}\label{tv:consist}
Let $g\in\overline{\spa{\Phi}}$ and assume that $u^\dagger\in\BV$ is the unique
$J$-minimising solution of \eqref{intro:lineqn}. Moreover, let
$\set{\alpha_k}_{k\in\N}$ and $\set{N_k}_{k\in\N}$ be as in Theorem
\ref{gen:consist} and define $\hat u_k = \hat u_{N_k}(\alpha_k)$. Then,
additionally to the assertions in Theorem \ref{gen:consist} we have that
\begin{equation*}
  \lim_{k\ra\infty} \norm{\hat u_k -u^\dagger}_{\Ls{p}} = 0\quad\text{ a.s.}
\end{equation*}
for every $1\leq p<2$.
\end{corollary}

\begin{proof}
  From Theorem \ref{gen:consist} it follows that $\set{\hat u_k}_{k\in\N}$ is
  bounded a.s. in $\L{2}$ and that each weak cluster point is a $J$-minimising
  solution of \eqref{intro:lineqn}. Since we assumed that $u^\dagger$ is the
  unique $J$-minimising solution of \eqref{intro:lineqn}, it follows that $\hat
  u_k \rightharpoonup  u^\dagger$ in $\L{2}$ a.s. and therefore also in $\L{p}$
  for each $1\leq p<2$.
  
  Since $\Omega$ is assumed to be bounded, it follows that $\L{2}$ is
  continuously embedded into $\L{1}$. Thus, it follows from Theorem
  \ref{gen:consist} that almost surely $\sup_{k\in\N} \norm{\hat u_k}_{\TV} <
  \infty$.  From the compact embedding $\BV\hookrightarrow \L{p}$ for $1\leq p<
  2$, it hence follows that $\set{\hat u_k}_{k\in\N}$ is compact in $\L{p}$.  Thus,
  the assertion follows, since weak and strong limits coincide.
\end{proof}

Unfortunately, the above embedding technique can not be used in order to improve
the convergence rate result in Theorem \ref{gen:thmrates} to strong
$\Ls{p}$-convergence and thus we have to settle for the general results in
Theorem \ref{gen:thmrates}.

We recall that a function $u \in \BV$ satisfies the source condition, if there
exists $\xi \in \ran(K^*)$ such that $\xi \in \partial J(u)$. It is important to
note, that in many applications the elements in $\ran(K^*)$ exhibit high
regularity such as continuity or smoothness. Thus it is of particular interest,
if such regular elements in $\partial J(u)$ exist. If $u$ is itself a smooth
function, application of Green's Formula and Example \ref{tv:subgr} yield
\citep[see also ][Lem.3.71]{frick:SchGraGroHalLen09}.

\begin{lemma}\label{tv:smoothsubgr}
Let $u\in C_0^1(\Omega)$ and set $E[u] = \set{x\in\Omega~:~ \nabla u(x) \not =
0}$. Assume that there exists $z\in C_0^1(\Omega, \R^2)$ with $\abs{z}\leq 1$
and 
\begin{equation*}
  z(x) = -\frac{\nabla u(x)}{\abs{\nabla u(x)}}\quad\text{ for }x\in E[u].
\end{equation*}
Then, $\xi:=\diw{z}\in \partial J(u)$.
\end{lemma}

In many applications (such as imaging) the true solution $u\in\BV$ is not
continuous, as e.g. if $u$ is the indicator function of a smooth set
$D\subset\Omega$. The following examples shows that in this case we still have
$\partial J(u)\cap C_0^\infty(\Omega) \not = \emptyset$. For the analytical
details we refer to \cite[Ex. 3.74]{frick:SchGraGroHalLen09}

\begin{example}\label{tv:char}
  Assume that $D\subset\Omega$ is a closed and bounded set with
  $C^\infty$-boundary $\partial D$ and set $u = \chi_D$. The outward unit-normal $n$ of $D$ then can
  be extended to a compactly supported $C^\infty$-vector field $z$ with
  $\abs{z}\leq 1$. %(cf. Figure \ref{tv:subgradfig}).
   Independent of the choice of the extension, we then have
   $\xi:=\diw{z}\in\partial J(u)$ and $\xi\in C^\infty_c(\Omega)$.
\end{example}

\begin{example}\label{tv:conv}
  We consider $\Omega = [0,1]^2$ and $V = \L{2}$. Moreover, we assume that
  $\mathcal{P}_2$ denotes the set of all dyadic partitions
  of $\Omega$ (cf. Example \ref{mrnorm:dyadicsetass}) and that $\Phi$ is
  the collection of indicator functions w.r.t. elements in $\mathcal{P}_2$.
  
  For a  function $k:\R^2 \ra \R$, we consider the
  \emph{convolution operator} on $U$ defined by
  \begin{equation*}
    (Ku)(x) = \int_{\R^2} k(x-y)\bar u(y) \dx\quad \text{ for }x \in \Omega
  \end{equation*}
  where $\bar u$ denotes the extension of $u$ on $\R^2$ by zero-padding. Assume
  further that $u^\dagger$ is the indicator function on a closed and bounded set
  $D\subset\Omega$ with $C^\infty$-boundary $\partial D$ and that $\xi \in
  \partial J(u^\dagger)$ is as in Example \ref{tv:char}. If the
  Fourier-transform $\mathcal{F}(k) =:\hat k$ of $k$ is non-zero a.e. in $\R^2$
  and if there exists $\beta\in(1,2]$ such that
  \begin{equation*}
    (1+\abs{\cdot}^2)^{-\beta\slash 2} \left(\hat \xi\slash \hat k\right) \in
    \Ls{2}(\R^2)\quad \text{ and }\quad \supp\left( p^\dagger :=
    \mathcal{F}^{-1}\left(\hat \xi\slash \hat k\right)\right) \subset \Omega,
  \end{equation*}
  then Assumption \ref{gen:assSE} is satisfied. To be more precise, we have
  that  $p^\dagger \in \mathcal{H}_{\beta-1}(\Omega)$ \cite[see][Thm.
  7.63]{frick:A75} and if there exists a constant $\kappa > 0$ such that
  $\alpha_k:=\sigma_k^{2\kappa}$ is summable it follows from Example
  \ref{lutz:dyadicpart} and Example \ref{tv:subgr} that
  \begin{equation*}
    \limsup_{k\ra\infty} \frac{\abs{\D\hat u_k}(\Omega) - \int_\Omega \xi \hat
    u_k \dx}{\sigma_k \sqrt{-\log \sigma_k}}
    =  \limsup_{k\ra\infty} \frac{\int_\Omega 1 - \cos(\gamma(\hat
    u_k, u^\dagger,x))\diff\abs{\D \hat u_k}(x)}{\sigma_k
    \sqrt{-\log \sigma_k }} < \infty\quad\text{ a.s.}
  \end{equation*}
  for the SMRE $\hat u_k = \hat u_{N_k}(\alpha_k)$ (where $N_k$ is as in Example
  \ref{lutz:dyadicpart}).
\end{example}

We close this section by two numerical examples that indicate the applicabillity
of SMRE for total variation based imaging. All SMREs are computed by an 
alternating direction method of multipliers (ADMM).  For implementation details
and further numerical comparisons see \cite{frick:FriMarMun11}. We note that, in
contrast to the theoretical considerations in Section \ref{gen}, the dictionary
$\Phi$ is usually fixed when computing SMRE for specific applications. Further,
the variance $\sigma^2$ is estimated from the data. Thus, the probability
$\alpha$ remains the only parameter to be chosen in the definition of the SMRE. 

\begin{example}\label{tv:ex_denois}
We first study the case of image denoising. We set $U = V = \R^{n\times n}$ with
$n=256$ equipped with the standard Euclidean inner product and induced norm and
study the model
\begin{equation*}
Y_{ij} = u^\dagger_{ij} + \sigma \eps_{ij},\quad1\leq i,j\leq n,
\end{equation*}
where $\eps = \set{\eps_{ij}}$ is a lattice of independent standard normal
random variables. We choose $u^\dagger$ to be the ``cameraman" image (with
values in $[0, 255]$). In the first column of Figure \ref{tv:denois} the
corresponding noisy images $Y$ are depicted with $\sigma = 30$ (upper row) and
$\sigma = 50$ (lower row).

Let $\mathcal{S}$ be the collection of all discrete squares in
$\set{1,\ldots,256}^2$ up to a maximal side length of $15$. Then, $\mathcal{S}$
consists of $N=930295$ elements and we choose the dictionary $\Phi$ to contain
all the scaled indicator functions $\phi_S = \chi_S\slash n$ for
$S\in \mathcal{S}$. Note that
\begin{equation*}
\norm{\phi_S} = \sqrt{\frac{\# S}{n^2}}\leq 1\quad \text{ and }\quad \phi_S^* =
\frac{\phi_S}{\norm{\phi_S}} = \frac{\chi_S}{\sqrt{\# S}}.
\end{equation*}
 where $\# S$ stands for the number of  grid-points in $S$. We choose the
 function $t_N(s,r)$ as in Section \ref{lutz}  (with $\gamma = d = 2$) such that
 the MR-statistic $T_N$ (Definition \ref{gen:mrstat}) takes the form
\begin{align*}
T_N(v) & = \max_{S\in\mathcal{S}} \abs{\inner{\phi^*_S}{v}} - \sqrt{-4\log
\norm{\phi_S}} \\
 & = \max_{S\in\mathcal{S}}  \left(\frac{1}{\sqrt{\# S}} \left|\sum_{(i,j)\in S}
 v_{ij} \right| - \sqrt{2 \log\left(\frac{n^2}{\# S}\right)}\right).
\end{align*}
Observe that this is the discrete version of the statistic \eqref{lutz:mrstat}.
Summarizing, for $\alpha\in[0,1]$ the SMRE $\hat u_N(\alpha)$ is a solution of
\begin{equation}\label{tv:smreden}
\inf_{u\in \R^{n\times n}} J(u)\quad\text{ s.t. }\quad
T_N(\hat\sigma^{-1}(Y-u))\leq q_N(\alpha)
\end{equation}
where $J$ denotes the discrete total variation functional and $q_N(\alpha)$ the
$1-\alpha$ quantile of $T_N(\eps)$. For the (presumably) unknown $\sigma$ we use
the estimator $\hat\sigma = 1.4826 \text{MAD}$, where $\text{MAD}$ denotes
the mean absolute deviation computed from the data $Y$. The solutions $\hat
u(\alpha)$ are depicted in the middle column of Figure \ref{tv:denois} together with the
normalized residuals $(Y-\hat u(\alpha))\slash \hat\sigma$ (last column).  For
all computations, we choose the quantile $q_N(\alpha) = -2$ in \eqref{tv:smreden}
that corresponds to a value $\alpha$ close to one. For both noise levels, the
residuals reveal hardly any non-random structure which confirms that the
MR-statistic constitutes a well-suited measure for data fidelty. Moreover, the
reconstructions are reasonably smooth while preserving local details. 
   
\begin{figure}[h!]

\includegraphics[height=0.31\textwidth]{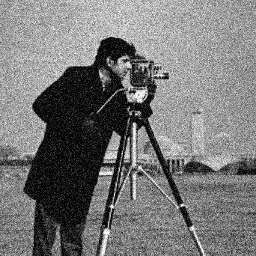}
\includegraphics[height=0.31\textwidth]{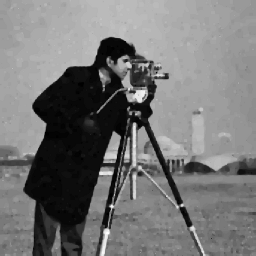}
\includegraphics[height=0.31\textwidth]{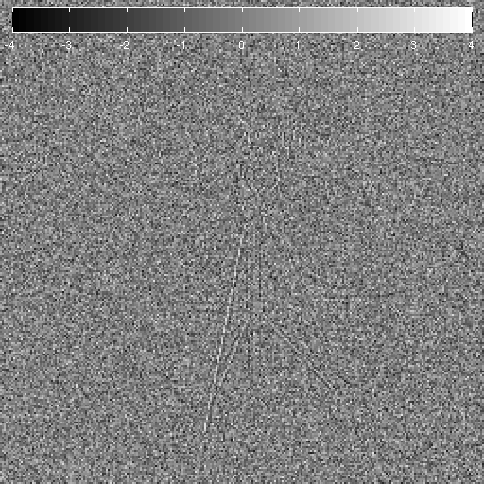}

\vspace{0.005\textwidth}

\includegraphics[height=0.31\textwidth]{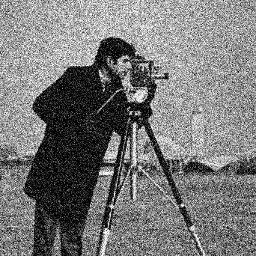}
\includegraphics[height=0.31\textwidth]{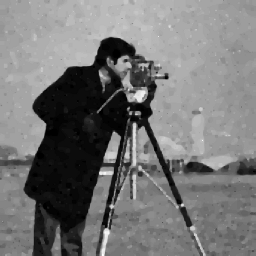}
\includegraphics[height=0.31\textwidth]{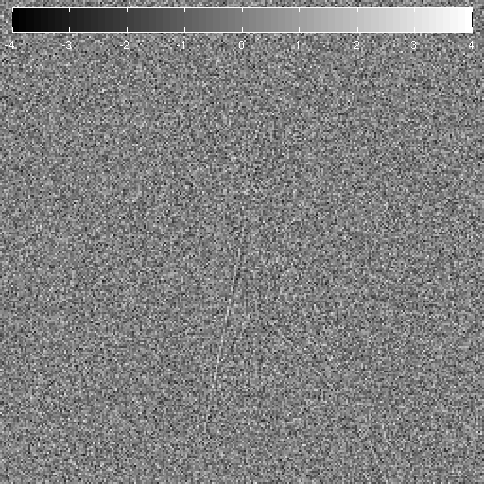}
\caption{First column: Data $Y$ with $\sigma = 30$ (top) and $\sigma = 50$
(bottom). Middle column: SMREs $\hat u (\alpha)$. Last column: Residuals
corresponding to the reconstructions in the middle column.}\label{tv:denois}
\end{figure}

To complement the visual impression, we compare our result with the recently
established locally adaptive image denoising method \cite{frick:Gra09}.  The
method requires a user-defined smoothing parameter $\theta\in[0,1]$ where  we
use the value $\theta = 0.6$ as suggested in \cite{frick:Gra09}. We note,  that
the procedure in \cite{frick:Gra09} does not require any a-priori knowledge on 
the variance $\sigma^2$ (In fact, it also applies to heteroscedastic noise).  
In order to gain a balanced assessment we compute three different types of distance measures
between estimator and the true signal: the signal-to-noise ratio (SNR) and the
integrated absolute error (IAE)
\begin{equation*}
\text{SNR}(u) = 10\log_{10}\left(\frac{\sum_{i,j}\abs{u_{ij}^\dagger - \bar
u^\dagger}^2}{\sum_{i,j}\abs{u^\dagger_{ij}- u_{ij}}^2} \right)\quad\text{
and }\quad \text{IAE}(u) =
\frac{1}{n^2}\sum_{i,j} \abs{u_{ij} - u^\dagger_{ij}},
\end{equation*}
where $\bar u^\dagger$ denotes the mean value of $u^\dagger$. SNR and IAE
basically measure the quality of the reconstruction in terms of the image
intensity. Additionally , we compute the Bregman distance
$D^{\xi}_J(\cdot,u^\dagger)$ (where we use a subgradient $\xi$ as in Example \ref{tv:smoothsubgr}) that measures the mean deviation between the unit normals at the level lines of the
reconstrunction and the true image (cf. Example \ref{tv:subgr}). The Bregman
distance hence measures how well the smoothness of the
reconstruction matches the smoothness of the true image. 

In Table \ref{tv:table} the averaged values of $100$ simulation runs for the
Bregman distance, SNR and IAE are listed for $\sigma = 30$ and $\sigma = 50$. 
\begin{table}[h!]
 \begin{tabular}{|l|p{1.5cm}p{1.5cm}p{1.5cm}|p{1.5cm}p{1.5cm}p{1.5cm}|}
 \hline
  & \multicolumn{3}{c|}{$\sigma = 30$} & \multicolumn{3}{c|}{$\sigma = 50$} \\
  \hline
   & Bregman & SNR & IAE & Bregman & SNR & IAE \\ 
  \hline
  SMRE  & 2.01 & 14.62 & 7.13 & 2.93 & 12.43 & 9.40 \\
  \cite{frick:Gra09}  & 2.30 & 13.52 & 7.57 & 3.57 & 11.96 & 9.62
  \\
  \hline 
 \end{tabular}
 \caption{Simulation results for ``cameraman'' image}\label{tv:table}
 \end{table}
As it can be seen from Table \ref{tv:table}, our approach outperforms the method
in \cite{frick:Gra09} with respect to all three distance measures. We mention
that SNR-values corresponding to other reconstruction methods can be found in
\cite{frick:Gra09} for a further comparison. 

 Finally, we stress that computation of
SMRE is numerically demanding: Whereas the estimators in
\cite{frick:Gra09} can be computed roughly in 10 seconds, the computation of
$\hat u(\alpha)$ takes up to 15 minutes (both Matlab implementations on a
dual-core (2.4GHz) computer). In the latter case, the computation time strongly
depends on the tolerance for numerical solutions of \eqref{tv:smreden} and on
the number of elements in the dictionary $\Phi$. We mention, though, that the
algorithmic methodology used in this example (see \cite{frick:FriMarMun11} for
details) permits efficient parallelisation which is not exploited in the current
implementation.
\end{example}

\begin{example}\label{tv:ex_deconv} 
Finally, we study the performance of the SMRE approach for image deconvolution,
i.e. we consider with $U$ and $V$ as in Example \ref{tv:ex_denois} the model 
\begin{equation*}
Y_{ij} = (Ku^\dagger)_{ij} + \sigma \eps_{ij},\quad1\leq i,j\leq n,
\end{equation*}
where $K$ is a convolution operator inducing \emph{motion blur} and where
$\sigma =13$. In Figure \ref{tv:deconv} the data (left image) and the SMRE
reconstruction $\hat u(\alpha)$ (middle image) are depicted, where $\hat u(\alpha)$ solves
\begin{equation*}
\inf_{u\in \R^{n\times n}} J(u)\quad\text{ s.t. }\quad
T_N(\hat\sigma^{-1}(Y-Ku))\leq q_N(\alpha).
\end{equation*}
The statistic $T_N$, the functional $J$, $q_N(\alpha)$ and $\hat \sigma$ are
chosen as in Example \ref{tv:ex_denois}. 

\begin{figure}[h!]
\includegraphics[height=0.31\textwidth]{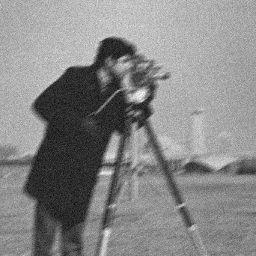}
\includegraphics[height=0.31\textwidth]{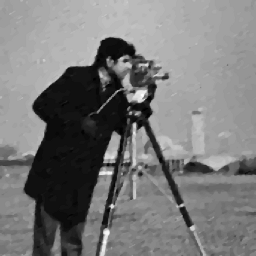}
\includegraphics[height=0.31\textwidth]{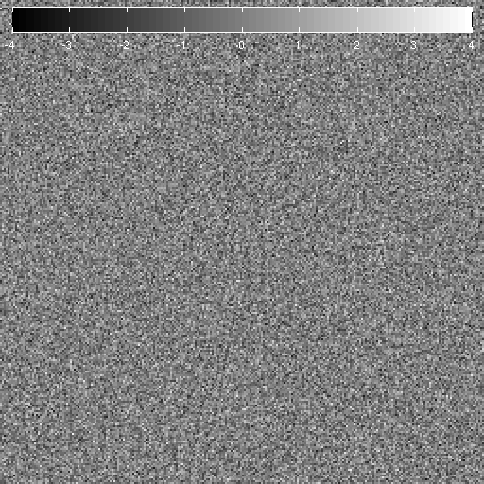}
\caption{Data $Y$ (left), SMRE $\hat u (\alpha)$ (middle) and
corresponding residuals.}\label{tv:deconv}
\end{figure}
 
The right image in Figure \ref{tv:deconv} shows the
standardised residuals. Similar as in the denoising case, the non-random
structures are reduced to a reasonable amount where at the same time the result
$\hat u(\alpha)$ does not seem to be underregularized. This gives numerical
evidence that SMREs are a promising approach for image deconvolution.
\end{example}

\section*{Acknowledgement}

K.F. and A.M are supported by the DFG-SNF Research Group FOR916
\emph{Statistical Regularization and Qualitative Constraints} (Z-Project). P.M
and A.M. are supported by the BMBF project $03$MUPAH$6$ \emph{INVERS} and by the SFB755
\emph{Nanoscale Photonic Imaging}. A.M. is supported by the SFB803
\emph{Functionality Controlled by Organization in and between Membranes}. 
The authors are indepted to L.~D{\"u}mbgen and A.~Tsybakov for stimulating
discussions and would like to thank two anonymous referees for their helpful
comments. 

%%%%%%%%%%%%%%%%%%%%%%%%%%%%%%%%%%%%%%%%%%
\appendix

\section{Source-condition and Bregman-divergence: Some examples}\label{app:breg}

The notions of source-condition and Bregman-divergence are very common in the
field of inverse problems. We will summarise the meaning of the
source-condition \eqref{defs:solution:source} and the Bregman-divergence for some
frequently used regularisation functionals $J$. 

\begin{example}\label{defs:norm}
  Let $J(u) = \frac{1}{2}\norm{u}^2$. Then, $J$ is differentiable on
  $U$ and for all $u\in U$ the set $\partial J(u)$ consists of the single
  element $\set{u}$. We have that $J'(v)(w) = \inner{v}{w}$ and consequently
  \begin{equation*}
    D_J(v,u) = D_J^\xi(v,u) = \frac{1}{2}\norm{v-u}^2\quad\text{ for }\xi =
    u\in \partial J(u).
  \end{equation*}
  Moreover, the source condition \eqref{defs:solution:source} can be rewritten to
  \begin{equation*}
    u^\dagger \in \ran(K^*).
  \end{equation*}
  Since $\ran(K^*) = \ran(K^*K)^{1\slash 2}$, this shows that the source
  condition \eqref{defs:solution:source} corresponds to the
  \emph{H\"older-source condition} $u^\dagger\in \ran(K^*K)^\beta$ for $\beta =
  1\slash 2$ \citep[see][]{frick:EHN96}. In
  \cite[Sec. 5.3]{frick:BisHohMunRuy07}, the H{\"o}lder-source condition w.r.t.
  a \emph{smoothing operator $K$} on Hilbert-scales has been discussed. To be
  more precise, assume that $\set{H_\mu}_{\mu \in \R}$ is a scale of Hilbert spaces and that $K$ is $a$-times smoothing, i.e.
  $K:H_{\mu - a} \ra H_\mu$ is continuous with continuous inverse. Then the
  condition $u^\dagger = (K^*K)^\beta p^\dagger$ implies that $u^\dagger \in
  H_{2a\beta}$. A prototype for Hilbert scales are Sobolev spaces. Here the
  index $\mu$ corresponds to the Sobolev index.
\end{example}

\begin{example}\label{defs:sparsity}
  Let $\set{\psi_n}_{n\in\N}$ be a ONB of $U$ and define
  \begin{equation*}
      J(u) = \norm{u}_1 := \sum_{j\in\N} \abs{\inner{u}{\psi_n}}.
  \end{equation*}
  In applications this functional promotes \emph{sparse solutions}, that is
  solutions that have only few non-zero coefficients w.r.t the basis
  $\set{\psi_n}_{n\in\N}$. As it was argued in \cite[Rem.
  17]{frick:GraHalSch08} the source-condition \eqref{defs:solution:source}
  holds if and only if there exist constants $a,b,\gamma > 0$ such that
  $\norm{u^\dagger}_1 <a$ and
  \begin{equation*}
      \bigl\|u\bigr\|_1 - \bigl\| u^\dagger \bigr\|_1 \geq -\gamma \bigl\|K(u -
      u^\dagger)\bigr\|
  \end{equation*}
  for all $u\in U$ such that $\norm{u}_1 < a$ and $\norm{K(u - u^\dagger)}<b$.
  If additionally for every finite set $J\subset \N$ the restriction of $K$ to
  the set $\overline{\set{\psi_n~:~ n\in J}}$ is injective, there exist
  constants $\beta_1,\beta_2 > 0$ such that
  \begin{equation*}
      \norm{u - u^\dagger}_1 \leq \beta_1 D_J^{K^*p^\dagger}(u,u^\dagger) +
      \beta_2 \norm{K(u - u^\dagger)}
  \end{equation*}
  for all $u\in U$  (see the proof of \cite[Thm. 15]{frick:GraHalSch08} and
  \cite[Thm 6.4]{frick:FriLorRes10}).
  \end{example}

\begin{example}\label{defs:h1norm}
  Assume that $U = \L{2}$ for an open and bounded set $\Omega\subset \R^n$
  with Lipschitz boundary $\partial \Omega$ and outer unit-normal $\nu$ and let
  $\H{\beta}$ denote the Sobolev-space of order $\beta\in\R$. We define
  \begin{equation*}
    J(u) = \begin{cases} \int_\Omega \abs{\nabla u}^2 \dx & \text{ if }u \in
\H{1} \\
+\infty & \text{ else.}
\end{cases} 
  \end{equation*}
  Then \citep[see][pp.63]{frick:B76}, the set $D(\partial J)$ consists of all
  elements $u \in \H{2}$ that have vanishing normal derivative $\inner{\nabla
  u}{\nu}$ on $\partial \Omega$ and if  $u\in D(\partial J)$, then
  $\partial J(u) = \set{-\Delta u}$. With this, it follows that $J'(v)(w) =
  \inner{\nabla v}{\nabla w}$ and
  \begin{equation*}
    D_J(v,u) = D_J^\xi(v,u) = \frac{1}{2}\norm{\nabla(v-u)}^2\quad\text{ for
    }\xi = -\Delta u \in \partial J(u). 
  \end{equation*}
  Moreover, $u^\dagger$ satisfies the source
  condition \eqref{defs:solution:source} with source element $p^\dagger\in V$ if
  and only if
  \begin{eqnarray*}
-(K^*p^\dagger)(x) & = & \Delta u^\dagger(x)\quad \text{ in }\Omega \\
\nabla u^\dagger \cdot \nu & = & 0 \quad \mathcal{H}^{n-1}\text{-a.e. on
} \partial \Omega
\end{eqnarray*}
(here $\mathcal{H}^{n-1}$ stands for the $(n-1)$-dimensional
Hausdorff-measure on $\partial \Omega$).
\end{example}

\begin{example}\label{tv:subgr}
Assume that $U$ is as in Example \ref{defs:h1norm} with $\Omega\subset \R^2$ and
let $J$ be the total variation semi-norm as defined in \eqref{tv:tv}. As it was
for example proved in \cite[Thm. 4.4.2]{frick:Fri08}, one has $\xi\in\partial
J(u)$ if and only if there exists
  $z\in\Ls{\infty}(\Omega, \R^2)$ with $\norm{z}_{\Ls{\infty}} \leq 1$ such that
  $\inner{z}{\nu}=0$ on $\partial \Omega$,
  \begin{equation*}
    \diw{z} = \xi\quad\text{ and }\quad \int_\Omega \xi u\dx = \abs{\D
    u}(\Omega).
  \end{equation*}
If $\xi \in \partial J(u)$, it thus follows that $D_J^\xi(v,u) = \abs{\D
v}(\Omega) -\int_\Omega \xi v \dx$. One can show that
\begin{equation*}
  D_J^\xi(v,u) = \int_\Omega (1- \cos(\gamma(v,u,x))) \diff \abs{\D v}(x)
\end{equation*}
where $\gamma(v,u,x)$ denotes the angle between the unit normals of the
sub-levelsets of $u$ and $v$ at the point $x\in\Omega$. 
\end{example}

% \begin{example}\label{defs:entropy}
%   Let $U$ be as in Example \ref{defs:h1norm} and define the \emph{negentropy} by
%   \begin{equation*}
%     J(u) = \begin{cases} -\int_\Omega u \log u \dx & \text{ if } u\geq
% 0\text{ a.e. and } u\log u\in \L{1} \\
% +\infty \text{ else.}
% \end{cases}
%   \end{equation*}
%   Then \citep[see][Chap. 2 Prop 2.7]{frick:BP78}, the set $D(\partial J)$
%   consists of all non-negative functions in $\L{\infty}$ that are bounded away
%   from zero. One has $J'(v)(w) = \inner{1+\log v}{ w}$ and if $u\in
%   D(\partial J)$, then $\partial J(u) = \set{1+\log u}$. After some
%   re-arrangements we find
%   \begin{equation*}
%     D_J(v,u) = D_J^\xi(v,u) = \int_\Omega \left( v \log\left(\frac{v}{u}\right)
%     - v + u\right) \dx,
%   \end{equation*}
%   that is, the Bregman-divergence coincides in this particular case with the
%   \emph{Kullback-Leiber-divergence}. It was proved in \cite[Lem.
%   2.2]{frick:BorLew91} that 
%   \begin{equation*} 
%     \norm{v-u}_{\Ls{1}}^2 \leq \left(\frac{2}{3} \norm{v}_{\Ls{1}} +
%     \frac{4}{3} \norm{u}_{\Ls{1}}\right) D_J(v,u).
%   \end{equation*}
%   In other words, Bregman-consistency (or convergence rates) w.r.t. the
%   negentropy yields strong consistency (convergence rates) in $\L{1}$. Finally,
%   we note that $u^\dagger\in D(\partial J)$ satisfies the source condition
%   \eqref{defs:solution:source} with source element $p^\dagger\in V$ if and only
%   if
%   \begin{equation*}
%     e^{(K^*p^\dagger)(x) - 1} = u^\dagger(x)\quad\text{ for a.e. }x\in\Omega.
%   \end{equation*}
% \end{example}
 
%%
%%
\section{Proofs}\label{app}  
\subsection{Proofs of the main results}\label{app:main}

In this section the proofs of the main results, that is  existence, consistency
and convergence rates for SMRE, are collected.  We start with a basic estimate
for the quantile function $q_N(\cdot)$ of the  MR-statistic as
defined in \eqref{gen:quantile}. We shall assume that
Assumptions \ref{defs:basicass} and \ref{gen:assK} hold.

\begin{lemma}\label{gen:quant}
Assume that $T_N$ is an MR-statistic and let $\alpha \in (0,1)$ and $N\in
\N$. Then,
\begin{equation*}
  q_N(\alpha) \leq \text{med}(T_N(\eps)) + L
  \sqrt{-2\log(2\alpha)}.
\end{equation*}
\end{lemma}

\begin{proof}
  First, we introduce the function $f(x_1,\ldots,x_N) =
  \max_{1\leq n\leq N} t_N(x_n, \norm{\phi_n})$. Then, $f$ is Lip\-schitz
  continuous with $\norm{f}_{\text{Lip}} \leq L$. Next, define for $1\leq n\leq N$ the random variables $\eps_n:=
  \eps(\phi_n^*)$. Then, $(\eps_1,\ldots,\eps_N) \sim \mathcal{N}(0,\Sigma)$ for
  a symmetric and positive matrix $\Sigma \in \R^{N\times N}$ with
  $\norm{\Sigma}_2 = 1$. Hence
  \begin{equation*}
    T_N(\eps) = \max_{1\leq n\leq N} t_N(\eps(\phi_n^*), \norm{\phi_n}) =
    f(\eps_1,\ldots,\eps_N) = f(\Sigma^{1\slash2} Z),
  \end{equation*}
  where $Z$ is an $N$-dimensional random vector with independent standard normal
  components. In other words, the statistic $T_N(\eps)$ can be written as the
  image of $Z$ under the Lipschitz function $f(\Sigma^{1\slash 2} \cdot)$.
  Applying Borel's inequality \citep[see][Lem. A.2.2]{frick:VaaWel96} we find
  that  $2\Prob\left(T_N(\eps) - \med(T_N(\eps)) > L \eta \right) \leq
   \exp\left( -(\eta^2\slash 2)\right)$ for all $\eta \in \R$. Now let 
  $\alpha\in (0,1)$, choose $q<q_N(\alpha)$ and set $\eta = (q -
  \med(T_N(\eps)))\slash L$. Then, $\Prob(T_N(\eps) \leq q) < 1-\alpha$ and
  hence
  \begin{equation*}
    \alpha  = 1 - (1-\alpha) < 1 - \Prob\left( T_N(\eps) < q \right)
     = \Prob\left( T_N(\eps) \geq  q \right) \leq \frac{1}{2}
    \exp\left(-\frac{1}{2} \left( \frac{q - \med(T_N(\eps))}{L}\right)^2\right).
  \end{equation*}
  Rearranging the above inequality yields 
  \begin{equation*}
  q < \med(T_N(\eps)) + L\sqrt{-2\log(2\alpha)},\quad\text{ for all }q<
  q_N(\alpha).
  \end{equation*}
  The assertion follows for $q\ra q_N(\alpha)$.
\end{proof}

We proceed with the proof of the existence result in Theorem
\ref{gen:smrexist}. To this end we use a standard compactness argument from
convex optimisation. For the sake of completeness, however, we will present the proof.

\begin{proof}[Proof of Theorem \ref{gen:smrexist}]
  Let $N \geq N_0$ and $y\in V$ be arbitrary. Due tu Assumption
  \ref{defs:basicass} (ii), $D(J)\subset U$ is dense and hence there exists for
  all given $\delta > 0$ an element $u_0 \in D(J)$ such that $\norm{Ku_0 - \tilde y} \leq\delta$, where $\tilde y$
  denotes the orthonormal projection of $y$ onto $\overline{\ran(K)}$. Since
  $\phi_n \in \overline{\ran(K)}$ and $\norm{\phi_n^*} = 1$ for all $n\in\N$,
  this implies that $\abs{\inner{Ku_0 - y}{\phi_n^*}} =  \abs{\inner{Ku_0 -
  \tilde y}{\phi_n^*}} \leq \delta$ for all $n\in\N$.
  
  Now let $\sigma > 0$ and
  $\alpha \in (0,1)$. Since $T_N$ is an MR-statistic (cf. Definition
  \ref{gen:mrstat}) we find that $t_N(0,r) < 0$ for all $r\in (0,1]$. Thus,
  according to according to the reasoning above, there exists $u_0 \in D(J)$
  such that for $1\leq n\leq N$
  \begin{equation}\label{aux1}
    L \sigma^{-1}\abs{y_n - \inner{Ku_0}{\phi_n^*}} \leq
    q_N(\alpha) - \max_{1\leq n\leq N} \lambda_N(\norm{\phi_n}),
  \end{equation}
  if the right-hand side is positive. To see this, assume that $q_N(\alpha) \leq
  \max_{1\leq n\leq N} \lambda_N(\norm{\phi_n})$. Since  for $1\leq n\leq N$ we
  have that  $t_N(\abs{\eps(\phi_n^*)}, \norm{\phi_n}) \geq
  \lambda_N(\norm{\phi_n})$ almost surely according to \eqref{gen:ineq}, it then
  follows that
  \begin{equation*}
    \Prob\left( T_N(\eps) \geq q_N(\alpha) \right) \geq \Prob\left(T_N(\eps)
    \geq  \max_{1\leq n\leq N} \lambda_N(\norm{\phi_n}) \right) = 1.
  \end{equation*}
  This is a contradiction to the definition of $q_N(\alpha)$ in
  \eqref{gen:quantile} and thus $u_0 \in D(J)$ as in \eqref{aux1} can be
  chosen. Since $s\mapsto t_N(s,r)$ is
  Lipschitz-continuous with constant $L$ and increasing for all $r\in (0,1]$,
  we find $t_N(\sigma^{-1} \abs{y_n - \inner{Ku_0}{\phi_n^*}}, \norm{\phi_n})
  \leq t_N(0,\norm{\phi_n}) + L \sigma^{-1} \abs{y_n - \inner{Ku_0}{\phi_n^*}}
    \leq q_N(\alpha)$ for $1\leq n\leq N$. 
  In other words, there exists at least one element $u_0 \in D(J)$ such that
  \begin{equation*}
    u_0 \in S := \set{u \in U~:~ \max_{1\leq n\leq N} t_N(\sigma^{-1}\abs{y_n
    - \inner{Ku}{\phi_n^*}}, \norm{\phi_n}) \leq q_N(\alpha) }.
  \end{equation*}
  Now, choose a sequence
  $\set{u_k}_{k\in\N} \subset S$ such that $J(u_k) \ra \inf_{u\in S} J(u)$. 
  This shows that $\sup_{k\in\N} J(u_k) =: a < \infty$. Moreover, we find from
  \eqref{gen:ineq}, that there exist constants $c_1, c_2 >0$
  such that for all $1\leq n \leq N$
  \begin{multline*}
    c_1 \sigma^{-1} \abs{y_n - \inner{K u_k}{\phi_n^*}} + c_2
    t_N(\abs{y_n - \inner{K u_k}{\phi_n^*}},\norm{\phi_n})  \\ \leq
    t_N(\sigma^{-1}\abs{y_n - \inner{K u_k}{\phi_n^*}},\norm{\phi_n}) \leq
    q_N(\alpha).
  \end{multline*}
  Together with \eqref{gen:lowbnd}, this shows $c_1 \sigma^{-1} \abs{ y_n -
  \inner{K u_k}{\phi_n^*}} + c_2 \lambda_N(\norm{\phi_n}) \leq q_N(\alpha)$. 
  Rearranging the inequality above yields
  \begin{equation*}
    \max_{1\leq n \leq N} \abs{\inner{Ku_k}{\phi_n^*}} \leq
    \max_{1\leq n\leq N} \abs{y_n} + \frac{\sigma}{c_1}
    \left( q_N(\alpha) - c_2\inf_{1\leq n\leq N} \lambda_N(\norm{\phi_n})
    \right)=: b < \infty.
  \end{equation*}
  Summarising, we find that $u_k \in \Lambda(a+b)$ for all $k\in\N$, as a
  consequence of which we can drop a weakly convergent  sub-sequence (indexed by
  $\rho(k)$ say) with weak limit $\hat u$. Since we assumed that $t_N(\cdot, r)$
  is convex for all $r\in (0,1]$, it follows that the admissible region $S$ is
  convex and closed and therefore weakly closed. This shows that $\hat u \in
  S$.   Moreover, the weak lower semi-continuity of $J$  (cf. Assumption
  \ref{defs:basicass} (ii)) implies
  \begin{equation*}
    J(\hat u) \leq \liminf_{k\ra\infty} J(u_{\rho(k)}) = \inf_{u\in S} J(u)
  \end{equation*}
  and the assertion follows with $\hat u_N(\alpha) = \hat u$
\end{proof}

In order to prove Bregman-consistency of
SMR-estimation in Theorem \ref{gen:consist}, we first establish a
basic estimate for the data error.

\begin{lemma}\label{gen:imgest}
Let $N\geq N_0$ and $\alpha\in (0,1)$. Moreover, assume that $u^\dagger$  is a
solution  of \eqref{intro:lineqn} and that $\hat u_N(\alpha)$ is an SMRE. Then,
for $1\leq n\leq N$
\begin{equation*}
  c_1 \sigma^{-1} \abs{\inner{Ku^\dagger - K\hat
  u_N(\alpha)}{\phi_n^*}} \leq  T_N(\eps) - 2 c_2
  \lambda_N(\norm{\phi_n}) +  \med(T_N(\eps)) +
  L \sqrt{-2\log(2\alpha)}.
\end{equation*}
\end{lemma}

\begin{proof}
  From Definition \ref{gen:smrest} it follows that $t_N( \sigma^{-1}
  \abs{\inner{ K u^\dagger - K\hat u_N(\alpha) + \sigma \eps}{\phi_n^*}}, 
  \norm{\phi_n}) \leq q_N(\alpha)$ for $1\leq n\leq N$. The convexity of $t_N$
  hence implies that
  \begin{multline*}
    t_N((2\sigma)^{-1} \abs{\inner{Ku^\dagger - K\hat u_N(\alpha)}{\phi_n^*}},
    \norm{\phi_n})  \\
    \leq \frac{1}{2}\left(t_N(\sigma^{-1}
    \abs{\inner{Y -  K\hat u_N(\alpha)}{\phi_n^*}}, \norm{\phi_n}) +
    t_N(\abs{\eps(\phi_n^*)}, \norm{\phi_n}) \right) \leq \frac{1}{2}
    ( q_N(\alpha) + T_N(\eps)).
  \end{multline*}
  By setting $v = (2\sigma)^{-1} \abs{\inner{Ku^\dagger - K\hat
  u_N(\alpha)}{\phi_n^*}}$ and $r = \norm{\phi_n}$ in \eqref{gen:ineq}, the
  above estimate shows that
  \begin{equation*}
    \frac{c_1}{2\sigma} \abs{\inner{Ku^\dagger - K\hat
    u_N(\alpha)}{\phi_n^*}} + c_2 t_N\left(\frac{1}{2}\abs{\inner{Ku^\dagger -
    K\hat u_N(\alpha)}{\phi_n}},\norm{\phi_n^*}\right) \leq \frac{q_N(\alpha) +
    T_N(\eps)}{2}.
  \end{equation*}
  Since $t_N(v,r) \geq \lambda_N(r)$ for all $v\in \R^+$ and $r\in (0,1]$
  (cf. \eqref{gen:lowbnd}) this implies for $1 \leq n\leq N$
  \begin{equation*}
  c_1\sigma^{-1} \abs{\inner{Ku^\dagger -
  K\hat u_N(\alpha)}{\phi_n^*}} \leq q_N(\alpha) + T_N(\eps) - 2c_2
    \lambda_N(\norm{\phi_n}).
  \end{equation*}  
  Finally, the assertion follows from Lemma \ref{gen:quant}.
\end{proof}

With these preparations, we are now able to prove Bregman-consistency.

\begin{proof}[Proof of Theorem \ref{gen:consist}]
  By the definition of the SMRE $\hat u_k =
  \hat u_{N_k} (\alpha_k)$, it follows that
  \begin{equation*}
    \Prob \left( J(\hat u_k) > J(u^\dagger) \right) \leq \Prob
    \left( T_{N_k}(\sigma_k^{-1}(Y - Ku^\dagger)) > q_{N_k}(\alpha_k)
    \right)  = \Prob \left( T_{N_k}(\eps) > q_{N_k}(\alpha_k) \right)
    \leq \alpha_k
  \end{equation*}
  for all $k\in \N$. Since $\sum_{k=1}^\infty \alpha_k <\infty$,
  it follows from the Borel-Cantelli Lemma \citep[see][p 255]{frick:Shi96} that
  $\Prob\left(J(\hat u_k) > J(u^\dagger)  \text{ i.o.} \right) \leq \Prob \left(
  T_{N_k}(\eps) > q_{N_k}(\alpha_k)\text{ i.o.} \right) = 0$, or in other words
  \begin{equation}\label{gen:Jleq}
    \Prob\left( \exists k_0 \in \N:~J(\hat u_k) \leq J(u^\dagger) \text{ for
    all } k\geq k_0 \right) = 1.
  \end{equation}
  In particular, it follows that $\sup_{k\in\N} J(\hat u_k) =: a < \infty$ a.s.
  
  Next, we note that $\sup_{N\in\N} T_N(\eps) < \infty$ a.s. implies that
  $\sup_{N\in\N} \med(T_N(\eps)) < \infty$. Hence, it follows from Lemma
  \ref{gen:imgest} and \eqref{gen:paramchoice} that $\max_{1\leq n\leq N_k}
  \abs{\inner{Ku^\dagger - K\hat
    u_k}{\phi_n^*}} = \bigo(\zeta_k)$ almost surely.
  as $k\ra\infty$ which proves \eqref{gen:imgconseqn}. In particular,
  \eqref{gen:imgconseqn} and the fact hat $N_k > N_0$ imply $\sup_{k\in\N}
  \max_{1\leq n\leq N_0} \abs{\inner{K\hat u_k}{\phi_n^*}} =: b < \infty$  a.s.
  Summarising, we find that $\hat u_k \in \Lambda(a+b)$ which is sequentially
  weakly precompact according to Assumption \ref{gen:assK} (ii). Choose a
  sub-sequence indexed by $\rho(k)$ with weak limit $\hat u \in U$. Since
  $N_k\ra\infty$ as $k\ra\infty$ it follows from \eqref{gen:imgconseqn} and
  \eqref{gen:paramchoice} that
  \begin{equation*}
    \abs{\inner{g - K\hat u}{\phi_n^*}} =  \lim_{k\ra\infty}
    \abs{\inner{Ku^\dagger  - K \hat u_{\rho(k)}}{\phi_n^*}} = 0\quad \text{
    for all  }n\in\N.
  \end{equation*}
  Since we assumed that $g\in \overline{\spa \Phi }$ this shows that $K\hat u =
  g$. Furthermore, according to \eqref{gen:Jleq} there exists (almost surely) an
  index $k_0$ such that $J(\hat u_{\rho(k)})$ does not exceed $J(u^\dagger)$ for
  all $k \geq k_0$. Together with the weak lower semi-continuity of $J$ this
  shows $J(\hat u)\leq \liminf_{k\ra\infty} J(\hat u_{\rho(k)}) \leq
    \limsup_{k\ra\infty} J(\hat u_{\rho(k)}) \leq J(u^\dagger)$.  Since
    $u^\dagger$ is a $J$-minimising solution of \eqref{intro:lineqn} we conclude
    that the same holds for $\hat u$ and that  $J(\hat u) = J(u^\dagger)=
    \lim_{k\ra\infty} J(\hat u_{\rho(k)}$.  In particular, for each sub-sequence
    $\set{J(u_k)}_{k\in\N}$ there exists a further sub-sequence that converges
    to $J(u^\dagger)$. This already shows that $\lim_{k\ra\infty} J(\hat u_k) =
    J(u^\dagger)$ a.s.
  
  We next prove that $D_J(u^\dagger, \hat u_k) \ra 0$. To this end, recall that
  there almost surely exists an index $k_0$ such that for $k\geq k_0$ one has
  $T_{N_k}(\eps) \leq q_{N_k}(\alpha_k)$. In order to exploit strong duality
  arguments, however, we have to make sure that the interior of the admissible
  region is non-empty (Slater's constraint qualification). But since we assumed
  that $s\mapsto t_N(s,r)$  is (strictly) increasing for each fixed $r\in (0,1]$
  it follows that $\Prob\left(t_{N_k}(\abs{\eps(\phi_n^*)}, \norm{
    \phi_n^*})=q_{N_k}(\alpha_k)\right) = 0$  for all $n\in\N$ and thus
  \begin{equation}\label{gen:consintnonempty}
    \Prob \left( \exists k_0:~ T_{N_k}(\eps) < q_{N_k}(\alpha_k)\text
    { for all }k\geq k_0 \right) = 1.
  \end{equation}
  By introducing the functional
  \begin{equation*}
    G_k(v) = \begin{cases}
0 & \text{ if } T_{N_k}(\sigma_k^{-1}(Y - v)) \leq
q_{N_k}(\alpha_k)  \\
+\infty & \text{ else,}
\end{cases}
  \end{equation*}
  we can rewrite \eqref{intro:opt} into $\hat u_k \in \argmin_{u\in U} J(u)
  + G_k(Ku)$. From \eqref{gen:consintnonempty} it follows that $u^\dagger$ lies
  in the interior of the admissible set of the convex problem \eqref{intro:opt}. In
  other words, the functionals $G_k$ are continuous at $Ku^\dagger$ for $k$
  large enough. Therefore we can apply \cite[Chap. II Prop. 4.1]{frick:ET76}
  (cf. also Chapter II, Remark 4.2 therein) and choose an element $\xi_k \in V$
  such that $K^*\xi_k \in \partial J(\hat u_k)$ and $-\xi_k
  \in \partial G_k(K \hat u_k)$. The second inclusion and the definition of the
  sub-gradient show that $G_k(Ku) \geq G_k(\hat u_k) - \inner{\xi_k}{Ku - K\hat
    u_k} = \inner{K^*\xi_k}{\hat u_k - u}$ for all
  $u\in U$. In particular, $u^\dagger$ satisfies $T_{N_k}(\sigma_k^{-1}(Y -
  Ku^\dagger)) = T_{N_k}(\eps) < q_{N_k}(\alpha_k)$ and thus $G_k(Ku^\dagger) =
  0$. This shows $0 \geq \inner{K^*\xi_k}{\hat u_k - u^\dagger}$. Since $J(\hat
  u_k)\ra J(u^\dagger)$ we find
  \begin{multline*}
    0\leq \limsup_{k\ra\infty} D_J(u^\dagger,\hat u_k) \leq
    \limsup_{k\ra\infty} D_J^{K^*\xi_k}(u^\dagger,\hat u_k)  \\
    =
    \limsup_{k\ra\infty}  J(u^\dagger) - J(\hat u_k) -
    \inner{K^*\xi}{u^\dagger - \hat u_k} \leq \limsup_{k\ra\infty}
    J(u^\dagger) - J(\hat u_k) = 0.
  \end{multline*}
  This proves \eqref{gen:conseqn}.
\end{proof}

It remains to prove the convergence rate results in Theorem \ref{gen:thmrates}.
To this end additional regularity of the true $J$-minimising solutions
$u^\dagger$ of \eqref{intro:lineqn} has to be taken into account. This
is formulated in Assumption \ref{gen:assSE}. With this we get the following
basic estimate.

\begin{lemma}\label{gen:imgrates}
Assume that Assumption \ref{gen:assSE} holds and let $N\geq N_0$ and $\alpha\in
(0,1)$. Then,
\begin{multline*}
  \abs{\inner{K^*p^\dagger}{\hat u_N(\alpha) - u^\dagger}} \leq
  \frac{\sigma}{c_1} \left( \tilde T_N(\eps) -
  2c_2 \inf_{1\leq n\leq N} \lambda_N(\norm{\phi_n}) + L \sqrt{-2\log(2\alpha)}
  \right)\sum_{n=1}^N \abs{b_{n,N}}  \\ + \rho_N \norm{K\hat u_N(\alpha) -
  K u^\dagger},
\end{multline*}
where $\tilde T_N(\eps) = T_N(\eps) +  \med(T_N(\eps))$.
\end{lemma}

\begin{proof}
  From Assumption \ref{gen:assSE} we find that
  \begin{equation*}
    \begin{split}
      \abs{\inner{K^*p^\dagger}{\hat u_N(\alpha) - u^\dagger}} & =
      \abs{\inner{p^\dagger}{K\hat u_N(\alpha) - Ku^\dagger}} \\
      & \leq \abs{\inner{\sum_{n=1}^N b_{n,N} \phi_n^*}{K\hat
      u_N(\alpha) - Ku^\dagger}} + \rho_N \norm{K\hat
      u_N(\alpha) - Ku^\dagger} \\
      & \leq \sum_{n=1}^N \abs{b_{n,N}} \max_{1\leq n \leq N}
      \abs{\inner{  \phi_n^*}{K\hat u_N(\alpha) - Ku^\dagger}}+ \rho_N \norm{K\hat
      u_N(\alpha) - Ku^\dagger}.
    \end{split}
  \end{equation*}
  From Lemma \ref{gen:imgest} it follows that
  \begin{equation*}
    \max_{1\leq n \leq N} \abs{\inner{  \phi_n^*}{K\hat u_N(\alpha) -
    Ku^\dagger}} \leq \frac{\sigma}{c_1}\left( \tilde T_N(\eps) -
    2c_2 \inf_{1\leq n\leq N} \lambda_N(\norm{\phi_n}) + L \sqrt{-2\log(2\alpha)}
    \right)
  \end{equation*}
  which shows the assertion.
\end{proof}

Combination of the auxiliary result in Lemma \ref{gen:imgrates} with Theorem
\ref{gen:consist} paves the way to the proof of Theorem \ref{gen:thmrates}. 

\begin{proof}[Proof of Theorem \ref{gen:thmrates}]
  First, observe that Assumption \ref{gen:assSE} and
  the definition of $\eta_k$ imply \eqref{gen:paramchoice}, that is, all
  assumptions in Theorem \ref{gen:consist} are satisfied. Therefore
  $\set{\hat u_k}_{k\in\N}$ is bounded almost surely and due to the continuity
  of $K$ we find that $\sup_{k\in\N} \norm{K\hat u_k - Ku^\dagger}  <
  \infty$ a.s. After setting $B:= \sup_{N\in\N} \sum_{n=1}^N \abs{b_{n,N}}$,
  which is finite according to Assumption \ref{gen:assSE}, it follows from Lemma
  \ref{gen:imgrates} and the definition of $\eta_k$ that
  \begin{equation}\label{gen:aux} 
    \abs{\inner{K^*p^\dagger}{\hat u_k - u^\dagger}} \leq
    \frac{B \sigma_k}{c_1}  \tilde T_{N_k}(\eps) + C \eta_k 
  \end{equation}
  for a suitably chosen constant $C>0$. Since $\sup_{N\in\N} T_N(\eps)<\infty$
  almost surely, it follows that also $\sup_{N\in\N} \tilde T_N(\eps) =
  \sup_{N\in\N}\left( T_N(\eps) + \med(T_N(\eps))\right) < \infty$ a.s.
  Combining this with \eqref{gen:aux} shows
  \begin{equation*}
    \abs{\inner{K^*p^\dagger}{\hat u_k - u^\dagger}} =
    \bigo (\eta_k)\quad\text{ a.s.}
  \end{equation*}
  Next, recall from \eqref{gen:Jleq} in the proof of Theorem \ref{gen:consist}
  that almost surely an index $k_0$ can be chosen such that for all $k\geq k_0$
  one has $J(\hat u_k) \leq J(u^\dagger)$. This shows that
  \begin{equation*}
    D_J^{K^*p^\dagger}(\hat u_k, u^\dagger) = J(\hat u_k) - J(u^\dagger) -
    \inner{K^*p^\dagger}{\hat u_k - u^\dagger}
    \leq \abs{\inner{K^*p^\dagger}{\hat u_k - u^\dagger}}
    = \bigo(\eta_k)
  \end{equation*}
  for $k\geq k_0$. This proves the first estimate in \eqref{gen:rateeqn}. The
  second estimate follows directly from Lemma \ref{gen:imgest}.
\end{proof}

\subsection{Approximation of continuous functions and entropy
estimates}\label{technical}

In this section we collect some results on the approximation
properties and entropy estimates for systems of piecewise constant functions
defined on a convex and compact set $\Omega\subset \R^d$ ($d\geq 1$). We start
with the following basic

\begin{definition}\label{app:hoelder}
Let $\Omega\subset \R^d$ be compact and convex.
\begin{enumerate}[(i)]
  \item For a function $g:\Omega\ra \R$, the \emph{modulus of continuity} is defined
  by
  \begin{equation*}
    \omega(\delta, g) = \sup_{\substack{s,t \in \Omega \\ \abs{s-t}_2\leq
    \delta}} \abs{g(s) - g(t)}\quad\text{ for }\delta > 0.
  \end{equation*}
  \item A function $g:\Omega\ra \R$ is called \emph{H\"older-continuous with
  exponent $\beta\in(0,1]$} if $\omega(\delta, g) = \bigo(\delta^\beta)$.
  The collection of all functions on $\Omega$ that are H\"older-continuous with
  exponent $\beta$ is denoted by $\mathcal{H}_\beta(\Omega)$.
\end{enumerate}
\end{definition}

The following lemma provides an error estimate for the  approximation of
a continuous $g:\Omega\subset \R^d\ra\R$ by piecewise constant functions in terms
of the modulus of continuity.

\begin{lemma}\label{reg:approx}
Let $\Omega\subset \R^d$ be a compact and convex set and $\set{A_1, A_2,\ldots}$
be a collection of measurable sub-sets of $\Omega$. Assume that there exists an
increasing sequence $\set{n_l}_{l\in\N}\subset \N$ with $n_0 = 0$ such that
\begin{enumerate}[(i)]
  \item for all $n_{l}+1\leq i< j \leq n_{l+1}$ one has $\abs{A_i \cap A_j}
  = 0$, 
  \item and $\Omega = A_{n_l + 1} \cup \ldots \cup A_{n_{l+1}}$
\end{enumerate}
for all $l\in\N$.
Then, for all continuous $g:\Omega\ra \R$ there exist coefficients $b_{j,l}^m$
such that
\begin{equation*}
  \sup_{m\in\N} \sum_{l=0}^{m} \sum_{j = n_l + 1}^{n_{l+1}} \abs{b^m_{j,l}}
  \leq \norm{g}_\infty\quad\text{ and }\quad \norm{g - \sum_{l=0}^{m} \sum_{j =
  n_l + 1}^{n_{l+1}} b^m_{j,l} \chi_{A_j}}^2 \leq \frac{m+1}{\sum_{l=0}^m
  \omega^{-2}(\delta_l, g)},
\end{equation*}
where $\delta_l := \max_{n_l < j \leq n_{l+1}}
\text{diam}(A_j)$.
\end{lemma}

\begin{proof}
  Let $g:\Omega\ra \R$ be continuous. For $l\in \N$ we define
  \begin{equation*}
    g_l = \sum_{j = n_l + 1}^{n_{l+1}} \abs{A_j}^{-1} \int_{A_j}
    g(\tau)\diff\tau\cdot \chi_{I_j}.
  \end{equation*}
  Next, we introduce $a_{lm} =
  (\omega^{-2}(\delta_l, g))\slash(\sum_{\nu=0}^m \omega^{-2}(\delta_\nu, g))$
  for $m\in\N$ and $1\leq l\leq m$. Note, that $a_{lm} \in (0,1)$ and
  $\sum_{0\leq l\leq m} a_{lm} = 1$. With this, we define for $0\leq l \leq m$ and $n_l < j \leq n_{l+1}$ the coefficients $b^m_{j,l} = (a_{lm}
  \int _{A_j} g(\tau) \diff\tau)\slash \abs{A_j}$. Since we
  assumed that $g$ is continuous on the compact set $\Omega$, it follows that
  $\abs{b^m_{j,l}} \leq \norm{g}_\infty a_{lm}$ and hence $\sum_{l=0}^{m}
  \sum_{j = n_l + 1}^{n_{l+1}} \abs{b^m_{j,l}} \leq \norm{g}_\infty$  for all
  $m\in\N$.
  Moreover, we have for all $s\in \Omega$ that
  \begin{equation*}
    \abs{\sum_{l=0}^m a_{lm} g_l(s) - g(s)} \leq  \sum_{l=0}^m a_{lm}\left(
    \sum_{j = n_l + 1}^{n_{l+1}} \frac{1}{\abs{I_j}}  \int_{A_j}
    \abs{g(\tau) - g(s)} \diff\tau \cdot \chi_{A_j}(s) \right).
  \end{equation*}
  After applying Jensen's inequality and keeping in mind that $\abs{s-t}
  \leq \delta_l$ for  $s,t \in A_j$ and $n_l < j \leq n_{l+1}$ it follows
  that
  \begin{equation*}
    \begin{split}
      \int_\Omega \abs{\sum_{l=0}^m a_{lm} g_l(s) - g(s)}^2 \diff s &
      \leq
      \sum_{l=0}^m a_{lm} \int_\Omega \left( \sum_{j = n_l + 1}^{n_{l+1}}
      \frac{1}{\abs{A_j}} \int_{A_j} \abs{g(\tau) - g(s)}^2
      \diff\tau \cdot \chi_{A_j}(s) \right)\diff s \\ &  =
      \sum_{l=0}^m a_{lm}
      \sum_{j = n_l + 1}^{n_{l+1}} \int_{A_j} \frac{1}{\abs{A_j}}
      \int_{A_j} \abs{g(\tau) - g(s)}^2 \diff \tau \diff s \\
      & \leq  \sum_{l=0}^m a_{lm} \omega^2(\delta_l, g)  \sum_{j = n_l +
      1}^{n_{l+1}} \abs{A_j}.
    \end{split}
  \end{equation*}
  Assumptions (i) and (ii) together with the definition of the coefficients
  $a_{lm}$ eventually yield
  \begin{equation*}
    \int_\Omega \abs{\sum_{l=0}^m a_{lm} g_l(s) - g(s)}^2 \diff s
    \leq \frac{m+1}{\sum_{\nu=0}^m \omega^{-2}(\delta_\nu, g)}.
  \end{equation*}
\end{proof}

For the remainder of this section we collect some results concerning the
capacity number of (subsystems of) the set $\Phi_d$ of indicator functions on
convex and closed sets in $[0,1]^d$ with $d\geq 1$. We first recall the basic
definition

\begin{definition}\label{lutz:capcovdef} Let $(T,d)$ be a semi-metric space, $T'
\subset T$ and $\eps > 0$. The \emph{capacity number} is defined by
\begin{equation*}
  D(\eps, T'):= \sup_{T''\subset T'}\left(\set{\# T''~:~ d(a,b)\geq
  \eps\text{ for all } a\not = b \in T'}\right).
\end{equation*}
\end{definition}

From a practical point of view, it is often more convenient to express
\eqref{mrnorm:caprequ} in terms of the \emph{$\eps$-covering number} $N(\eps,
T')$ of $T'$ which is defined as the smallest number of $\eps$-balls in $T$
needed to cover $T'$ (the center points need not to be elements of $T'$,
though).  It is common knowledge \citep[see][p.98]{frick:VaaWel96} that for all $\eps > 0$
\begin{equation}\label{mrnorm:capcov}
  N(\eps, T) \leq D(\eps, T) \leq N(\eps\slash 2, T).
\end{equation}
% In order to see the equivalence of capacity and covering number, fix
% \begin{equation*}
%   T'' \subset T' \text{ with } D(\eps,T') = \# T'' \text{ and } d(a,b) >
%   \eps \text{ for all } a,b \in T'', a \neq b.
% \end{equation*}
% If there was $a_0 \in T'$ such that
% $a_0 \not \in \bigcup_{a \in T''} B_{\eps}(a)$ it would follow that $d(a_0,
% a) > \eps$ for all $a \in T''$. This contradicts the minimality of $T''$ with
% that property. On the other hand, any ball $B_{\eps / 2}(a)$ with center $a
% \in T$ contains at most one element of $T''$. We therefore find that
% %   Indeed, for a semi-metric space $(T,d)$ and $\eps > 0$ we find that
%
% For further details on that topic we refer to \citep[Chap.
% 2.2]{frick:VaaWel96}.

We consider $\Phi_d\subset \Ls{2}([0,1]^d)$ as a metric space with the induced
$\Ls{2}$-metric, i.e. for $\chi_P,\chi_Q \in \Phi_d$ we have
\begin{equation*}
  d(\chi_Q, \chi_P)^2 = \norm{\chi_P - \chi_Q}^2 = \int_{[0,1]^d} (\chi_Q -
  \chi_P)^2 \diff \dx= \abs{Q \triangle P}.
\end{equation*}
The entire set $\Phi_d$ is too large in order to render the test-statistic $T_N$
in \eqref{lutz:statistic} finite: It was shown in \cite{frick:Bro76} 
\citep[see also][Chap. 8.4]{frick:dud99}) that the $\eps$-covering number of
$\Phi_d$ of all nonempty, closed and convex sets contained in the unit ball $\set{x\in \R^d~:~
\abs{x} \leq 1}$ is of the same order as $\exp(\eps^{(1-d)\slash 2})$ (for $d\geq
2$)  as $\eps\ra 0^+$. This proves that there cannot exist any constants $A$, $B$
and $\gamma$ such that \eqref{mrnorm:caprequ} holds with $\Phi = \Phi_d$.
% \begin{equation}\label{app:genconventr}
%   D(u\delta, \set{\phi\in\Phi_d~:~ \norm{\phi}\leq \delta}) \leq
%   A u^{-B} \delta ^{-\gamma}\quad \text{ for all }u,\delta\in(0,1].
% \end{equation}

For particular classes of convex sets, however, entropy estimates
as in \eqref{mrnorm:caprequ} are at hand. The collection $\Phi_{r}$ of
indicator functions on $d$-dimensional rectangles in $[0,1]^d$ constitutes
such an example:

\begin{proposition}\label{app:entrrect}
There exists a constant $A = A(d) > 0$ such that
\begin{equation*}
  D(u\delta, \set{\phi\in\Phi_r~:~ \norm{\phi}\leq \delta}) \leq A
  (u\delta)^{-4d}
\end{equation*}
for all $u,\delta\in(0,1]$.
\end{proposition}

\begin{proof}
  From \cite[Thm. 2.6.7]{frick:VaaWel96} it follows that the $\eps$-covering
  number of $\Phi_r$ can be estimated by $A \eps^{-2(V-1)}$ where $V$ denotes
  the VC-index of the set of subgraphs
  $\set{(x,t)~:~ t < \phi(x)}$ for $\phi\in \Phi_r$. This in turn is equal to
  the VC-index of the collections of all rectangles in $[0,1]^d$ which is
  $2d+1$ \citep[see][Ex. 2.6.1]{frick:VaaWel96}.
\end{proof}

For certain subsets of $\Phi_r$ better estimates can be derived. We close this
section with results for the system $\Phi_s$ and $\Phi_2$ of indicator functions
on all squares and dyadic partitions in $[0,1]^d$ respectively. We skip the
proofs, for they are elementary but rather tedious.

\begin{proposition}\label{mrnorm:allsetass}
There exists a constant $A = A(d) > 0$ such that
\begin{equation*}
  D(u\delta, \set{\phi\in\Phi_s~:~ \norm{\phi}\leq \delta}) \leq A u^{-2(d+1)}
  \delta^{-d},\quad \text{
  for all }u,\delta \in(0,1].
\end{equation*}
\end{proposition}

\begin{proposition}\label{mrnorm:dyadicsetass}
Let $d\geq 2$ and consider the system of all
\emph{dyadic partitions} in $[0,1]^d$, that is
\begin{equation*}
  \mathcal{P}_{2} := \set{Q\subset [0,1]^d~:~ Q = 2^{-k}(i+[0,1]^d),\ k\in\N,
  i = (i_1,\ldots,i_d)\in\N^d }.
\end{equation*}
Let $\Phi_2$ the set of all indicator functions on elements in
$\mathcal{P}_2$. Then, there exists a constant $A = A(d) > 0$ such that
\begin{equation*}
  A^{-1} u^{-2} \delta^{-2} \leq D(u\delta, \set{\phi \in
  \Phi_2~:~ \norm{\phi} \leq \delta}) \leq A u^{-2} \delta^{-2},\quad \text{
  for all }u,\delta \in(0,1].
\end{equation*}
\end{proposition}

\bibliographystyle{abbrv}              
\bibliography{frick}   

\end{document}